\documentclass{amsart}

\usepackage{graphicx,amssymb,mathrsfs,amsmath,amsfonts,appendix}
\usepackage{paralist}
\usepackage{enumitem}
\usepackage[foot]{amsaddr}
\usepackage[top=1.5in, left=1.25in, right=1.25in, bottom=1.5in]{geometry}

\usepackage{psfrag}
\usepackage[colorinlistoftodos]{todonotes}

\renewcommand{\figurename}{Figure}
\newcommand{\figref}[1]{\figurename~\ref{#1}}

\newtheorem{theorem}{Theorem}[section]
\newtheorem{definition}[theorem]{Definition}
\newtheorem{lemma}{Lemma}[section]
\newtheorem{corollary}{Corollary}[section]

\title[Cross-diffusion with excluded-volume effects]{Cross-diffusion systems with excluded-volume effects and asymptotic gradient flow structures}

\author{Maria Bruna$^1$}
\address{$^1$Mathematical Institute, University of Oxford,  RQQ, Woodstock Road, Oxford OX2 6GG, UK}

\author{Martin Burger$^2$}
\address{$^2$Institut f\"ur Numerische und Angewandte Mathematik and Cells in Motion Cluster of Excellence, Westf\"alische Wilhelms Universit\"at M\"unster, Einsteinstrasse 62, D 48149 M\"unster, Germany}

\author{Helene Ranetbauer$^3$}
\address{$^3$RICAM, Austrian Academy of Sciences, Altenbergerstr. 63, 4040 Linz, Austria}

\author{Marie-Therese Wolfram$^4$}
\address{$^4$University of Warwick, Coventry CV4 7AL, UK and RICAM, Austrian Academy of Sciences, Altenbergerstr. 63, 4040 Linz, Austria}

\begin{document}

\label{firstpage}
\maketitle

\begin{abstract}
In this paper we discuss the analysis of a cross-diffusion PDE system for a mixture of hard spheres, which was derived  in \cite{Bruna:2012wu} from a stochastic system of interacting Brownian particles using the method of matched asymptotic expansions. The resulting cross-diffusion system is valid in the limit of small volume fraction of particles. While the system has a gradient flow structure in the symmetric case of all particles having the same size and diffusivity, this is not valid in general. We discuss local stability and global existence for the symmetric case using the gradient flow structure and entropy variable techniques. For the general case, we introduce the concept of an asymptotic gradient flow structure and show how it can be used to study the behavior close to equilibrium.  Finally we illustrate the behavior of the model with various numerical simulations.
\end{abstract}

\section{Introduction} \label{sec:intro}

Systems of interacting particles can be observed in biology (e.g. cell populations), physics or social sciences (e.g. animal swarms or large pedestrian crowds). Macroscopic models describing the individual interactions of these particles among themselves as well as their environment lead to complex systems of differential equations (cf. e.g. \cite{bendahmane2009conservative,Bruna:2012wu,Bruna:2012cg,MR2745794,burger:2015uk,burger2012nonlinear,di2016nonlocal,painter2009continuous,Schlake:2011wr,Simpson:2009gi}). In microscopic models the dynamics of each particle is accounted for explicitly, while the macroscopic models typically consist of partial differential equations for the population density. Passing from the microscopic model to the macroscopic equations in a systematic way is, in general, very challenging, and often one relies on closure assumptions, which can be made rigorous under certain scaling assumptions on the number and size of particles. In particular, when crowding due to the finite-size of particles is included in the model, the limiting process is quite subtle and, using different assumptions and closure relations, a variety of macroscopic equations have been derived. For instance, the macroscopic equations of a two-species system where particles undergo a simple exclusion process on a lattice can be derived using formal Taylor expansions, see for example \cite{MR2745794,Simpson:2009gi}. The case when particles are not confined to a regular lattice and undergo instead a Brownian motion with hard-core interactions was considered in \cite{Bruna:2012wu} using matched asymptotic expansions. Cross-diffusion is a common feature of all these models and poses a particular
challenge for the analysis since maximum principles do not hold. Classical examples of cross-diffusion systems are reaction diffusion systems or systems describing multicomponent gas mixtures. These quasi-linear parabolic systems were analyzed by Ladyzenskaya \cite{ol1968linear} or Amann \cite{amann1985global,amann1989dynamic}, which however rely on strong parabolicity assumptions that break down for the degenerate cross-diffusion systems derived from the interacting particle systems mentioned above.

The canonical form for a two-species system of interacting
particles (called red and blue in the following) is
\begin{align}\label{cross-diffusion}
\partial_t
\begin{pmatrix} r\\ b \end{pmatrix}
= \nabla \cdot \left(D(r,b) \nabla \begin{pmatrix} r \\ b \end{pmatrix} - F(r,b) \begin{pmatrix} r \\ b \end{pmatrix}\right),
\end{align}
where $D = D(r,b)$ is the diffusion matrix and $F = F(r,b)$ is the drift matrix due to a convective flux. 

Systems like \eqref{cross-diffusion} often have a gradient flow structure
\begin{align} \label{gradflow}
\partial_t \begin{pmatrix} r\\  b \end{pmatrix}    = 
\nabla \cdot \left[   M (r,b) \nabla 
\begin{pmatrix} \partial_r E  \\ \partial_b E \end{pmatrix}\right],
\end{align}
where $M$ is a mobility matrix and $\partial_r E$ and $\partial_b E$ denote the functional derivative of an entropy functions $E$ with respect to $r$ and $b$, respectively. 
The gradient flow formulation provides a natural framework to study the analytic behavior of such systems, cf. e.g. \cite{ambrosio2008gradient}.  It has been used to analyze existence and long-time behavior of systems, see for example \cite{carrillo2014gradient,jungel2014boundedness,liero2013gradient,zinsl2015transport}.
As a result, being able to express a PDE system as gradient flows of an entropy is a very desirable feature; yet, this is not possible in general. The lack of the gradient flow structure on the PDE level can result from the approximations made when passing from the microscopic description to the macroscopic equations. This is the case of the cross-diffusion system derived in \cite{Bruna:2012wu}, which was derived using the method of matched asymptotics. There has been a lot of research on the passage from microscopic models to the continuum equations, for example in the hydrodynamic limit \cite{kipnis2013scaling}. More recently the microscopic origin of entropy structures, which connects gradient flows and the large deviation principle was analyzed in \cite{adams2011larg,liero2015microscopic}. 

In this paper we introduce the idea of an asymptotic gradient flow structure as a generalization of a standard or full gradient flow for systems derived as an asymptotic expansion such as that in \cite{Bruna:2012wu}.
In this paper we provide several analytic results for these cross-diffusion systems and introduce the notion of asymptotic gradient flows. We discuss how the closeness of
these asymptotic gradient flow structures can be used to analyze the behavior of the system close to equilibrium. Furthermore we present a global in time existence result 
in the case of particles of same size and diffusivity (in which the system has a full gradient flow structure). The existence proof is based on an implicit Euler discretisation
and Schauder's fixed point theorem. We study the linearized system with an additional regularization term in the entropy to ensure boundedness of the solutions and deduce existence
results for the unregularised system in the limit (similar to the deep quench limit for the Cahn Hilliard equation \cite{Elliott:1996}). This is, to the authors' knowledge,
the first global in time existence result for this system so far. We note however that it is only valid if the total density stays strictly below the maximum density.

This paper is organized as follows: we introduce the mathematical model in Section \ref{sec:model} and discuss the cases for which the system has either a full or an asymptotic gradient flow structure. In Section \ref{sec:closetoequilibrium} we define the notion of asymptotic gradient flows formally and discuss how they can
be used to analyze the behavior of stationary solutions close to equilibrium. Several numerical examples illustrating the deviation of stationary solutions from the equilibrium solutions
for asymptotic gradient flows are presented in Section \ref{sec:numerics}. Finally, we give a global in time existence result in the case of particles of same size and diffusivity in Section \ref{sec:existence}.

\section{The mathematical model} \label{sec:model}

In this paper we analyze a cross-diffusion system for a mixture of hard spheres derived in \cite{Bruna:2012wu}, which we present below. The system is obtained as the continuum limit of a stochastic system with two types of interacting Brownian particles, referred to as red and blue particles. In particular, we consider $N_r$ red particles of diameter $\epsilon_r$, constant diffusion coefficient $D_r$ and external potential $\tilde V_r$,  and $N_b$ blue particles of diameter $\epsilon_b$,  diffusion coefficient $D_b$, and external potential $\tilde V_b$. Each particle evolves according to a stochastic differential equation (SDE) with independent Brownian dynamics, and interacts with the other particles in the system via hard-core collisions. This means that the centers of two particles with position ${\bf X}_i$ and ${\bf X}_j$ in space are not allowed to get closer than the sum of their radii, that is, $ \| {\bf X}_i - {\bf X}_j \| \ge (\epsilon_i + \epsilon_j)/ 2$, where $\epsilon_i$ denotes the radius of the $i$th particle. We define the total number of particles in the system by  $N = N_r + N_b$, and the distance at contact between a red and blue particles by $\epsilon_{br}=(\epsilon_r + \epsilon_b)/2$. 
The situation detailed above can be described by the overdamped Langevin SDEs
\begin{align} \label{sde}
\begin{aligned}
d {\bf X }_i (t) &=  \sqrt{2D_r}\, d{\bf W}_i(t)  - \nabla \tilde V_r( {\bf X}_i) \, dt \qquad 1\leq i\leq N_r,\\
d {\bf X }_i (t) &=  \sqrt{2D_b}\, d{\bf W}_i(t)  - \nabla \tilde V_b( {\bf X}_i) \, dt \qquad N_r+1 \leq i \leq N, 
\end{aligned}
\end{align}
where ${\bf X}_i \in \Omega \subset \mathbb R^d$, $d= 2, 3$, is the position of the $i$th particle and ${\bf W}_i$ a $d$-dimensional standard Brownian motion. We assume that $\Omega$ is a bounded domain. The boundary conditions due to collisions between particles and with the domain walls are
\begin{align}
\begin{aligned}
(d {\bf X }_i  - d {\bf X }_j) \cdot {\bf n} & = 0, & \quad &\text{on} \quad \| {\bf X}_i - {\bf X}_j \| = (\epsilon_i + \epsilon_j)/ 2,\\
d {\bf X }_i \cdot {\bf n} & = 0, & \quad &\text{on} \quad  \partial \Omega,
\end{aligned}
\end{align}
where $ \bf n $ denotes the outward unit normal.
The continuum-level model associated to this individual-based model was derived in \cite{Bruna:2012wu} using the method of matched-asymptotic expansions in the limit of low but finite volume fraction. If $v_d(\epsilon)$ is the volume of a $d$-dimensional ball of diameter $\epsilon$, then the volume fraction in the system is 
\begin{equation} \label{vol_fraction}
\Phi=  N_r v_d(\epsilon_r) + N_b v_d(\epsilon_b),
\end{equation}
assuming that the problem is nondimensionalised such that the domain $\Omega$ has unit volume, $|\Omega| = 1$. 
Because particles cannot overlap each other, in addition to the global constraint $\Phi \ll 1$ there is also a local constraint on the total volume density, defined as
\begin{equation} \label{total_vol_density}
\phi({\bf x}, t) = v_d(\epsilon_r) r({\bf x}, t) + v_d(\epsilon_b) b({\bf x}, t). 
\end{equation}
In particular, the local volume density cannot exceed the theoretical maximum allowed volume fraction, given by the Kepler conjecture. We note that $\Phi$ and $\phi$ are related via $\Phi = \int_\Omega \phi \, d {\bf x}$. 
 
The cross-diffusion model in \cite{Bruna:2012wu} is valid for any number of blue and red particles, $N_b$ and $N_r$. However, here we will consider the case that the number of both particles is large, such that $N_r-1 \approx N_r$,$N_b -1 \approx N_b$, as it simplifies the model slightly. In this case, the model reads \cite{Bruna:2012wu} 
\begin{subequations}
\label{pde_general}
\begin{align}
\label{pde_general_r}
\partial_t r   &=  D_r \nabla  \cdot \left[   ( 1 +
 \epsilon_r^d \alpha   r ) \nabla {  r} + \nabla V_r  r   +  \epsilon_{br}^d \big ( \beta_r \, { r}  \nabla { b} - \gamma_r { b}\nabla {  r}  +  \nabla ( \gamma_b  V_b - \gamma_r  V_r )  r  b\big ) \right],   
\\
\label{pde_general_b}
\partial_t b   &=  D_b \nabla   \cdot \left[   ( 1 +  \epsilon_b^d \alpha   b ) \nabla {  b} + \nabla V_b  b  +  \epsilon_{br}^d \big ( \beta_b \, { b}  \nabla { r} - \gamma_b { r}\nabla {  b}  +  \nabla \big( \gamma_r  V_r -  \gamma_b  V_b \big)  r  b\big ) \right],   
\end{align}
\end{subequations}
where $r=r({\bf x},t)$ and $b=b({\bf x},t)$ are the number densities of the red and blue species, respectively, depending on space and time. Consequently, meaningful solutions satisfy $r \ge 0$ with $\int_\Omega r\,  d{\bf x} = N_r$ and $b \ge 0$ with $\int_\Omega b \, d{\bf x} = N_b$.   In \eqref{pde_general},  $V_i = \tilde V_i/D_i$ are the rescaled potentials, and the parameters $\alpha$, $\beta_i$ and $\gamma_i$ depend on the geometry of the particles. For balls, they are given by
\begin{equation}
\label{coef_23d}
\begin{aligned}
\alpha &= \frac{2(d-1)\pi}{d} , \qquad   \beta_i =    \frac{2\pi}{d}  \frac{ [(d-1)D_i  +  dD_j]}{D_i + D_j}  ,  \qquad \gamma_i=  \frac{2\pi}{d} \frac{D_i}{D_i + D_j}  ,
\end{aligned}
\end{equation}
for $i = r$ or $b$, and space dimension $d=2$ or $3$. This system is an asymptotic expansion in $\epsilon_r$, $\epsilon_b$ (assuming that both small parameters are of the same asymptotic order,  $\epsilon_r \sim \epsilon_b \sim \epsilon$), valid up to order $\epsilon^d$. The nonlinear terms in \eqref{pde_general} correspond to the leading-order contribution of the pairwise particle interactions.  The asymptotic method used in \cite{Bruna:2012wu} could be extended if desired to evaluate higher-order terms coming from three or more particle interactions, as well as higher-order corrections in the pairwise interaction. This would result in higher-order terms in $\epsilon_i$ in \eqref{pde_general} (of order $\epsilon_i^{(d+1)}$ and higher) with quite some effort. However, it seems impossible to derive the full infinite series expansion. 

We will consider the system \eqref{pde_general} in $\Omega \times (0, T)$ with no-flux boundary conditions
\begin{subequations}
\label{bcs_general}
\begin{align}
\label{bcs_general_r}
0 & = {\bf n}  \cdot \left \{   ( 1 +
 \epsilon_r^d \alpha   r ) \nabla {  r} + \nabla V_r  r   +  \epsilon_{br}^d \big [ \beta_r \, { r}  \nabla { b} - \gamma_r { b}\nabla {  r}  +  \nabla ( \gamma_b  V_b - \gamma_r  V_r )  r  b\big ] \right \},   
\\
\label{bcs_general_b}
0 & = {\bf n}   \cdot \left \{   ( 1 +  \epsilon_b^d \alpha   b ) \nabla {  b} + \nabla V_b  b  +  \epsilon_{br}^d \big [ \beta_b \, { b}  \nabla { r} - \gamma_b { r}\nabla {  b}  +  \nabla \big( \gamma_r  V_r -  \gamma_b  V_b \big)  r  b\big ] \right \},   
\end{align}
\end{subequations}
on $\partial \Omega \times (0, T)$ and initial values
\begin{equation}\label{initial_general}
r({\bf x}, 0) = r_0({\bf x}), \qquad b({\bf x}, 0) = b_0({\bf x}).
\end{equation}

In order to analyze the cross-diffusion system \eqref{pde_general}, it is convenient to consider its associated gradient flow structure of the form \eqref{gradflow}. However, only the system in the  symmetric case where red and blue particles have same size and diffusivity can be rewritten in that form. For the general case we introduce a generalization of a gradient flow, namely an \emph{asymptotic gradient flow}, motivated by the underlying structure of the general system \eqref{pde_general}.

\subsection{Cross-diffusion system for particles of the same size and diffusivity} \label{sec:case1}

In this section we suppose that red and blue particles are of the same size, that is $\epsilon_r = \epsilon_b : = \epsilon$, and have the same diffusion coefficient, $D_r = D_b$. Without loss of generality, we take the diffusion coefficient equal one (this can be achieved by rescaling time). In this case, the cross-diffusion system \eqref{pde_general} can be written as
\begin{subequations}
\label{case1}
\begin{align}
\label{case1_r}
\partial_t r   &=  \nabla  \cdot \left[ (1 + \alpha \epsilon^d 
      r -  \gamma \epsilon^d { b} ) \nabla {  r}   + \beta \epsilon^d r  \nabla  b +  r  \nabla V_r + \gamma \epsilon^d \nabla \left (  V_b -  V_r \right )  r b \right] ,   
\\
\label{case1_b}
\partial_t b   &=  \nabla  \cdot \left[ ( 1 + \alpha  \epsilon^d b - \gamma \epsilon^d  r )  \nabla   b  + \beta \epsilon^d  b   \nabla  r  +  b  \nabla V_b + \gamma \epsilon^d \nabla \left ( V_r -  V_b \right )   r  b \right],
\end{align}
\end{subequations}
where $\beta_i$ and $\gamma_i$, for $i = r, b$, are now equal and simplify to $\gamma = \pi/d$ and $\beta = 2^{d-1}\gamma$, respectively. 


This cross-diffusion system can be used to describe a mixture of particles that are physically identically but that are driven by different potentials $V_r$ and $V_b$ (for example cells that are attracted to different food sources, or pedestrians that want to move in different directions). Moreover, it can also be used to model the scenario where the red and the blue particles are in fact identical, but one has knowledge about the initial distributions of each sub-population, $r_0$ and $b_0$. This is the scenario in many experimental set-ups that use noninvasive fluorescent tagging. 
On the other hand, if the red and blue particles are identical and initially indistinguishable, then one has that $r/N_r = b/N_b := p$ for all times. In this case, both equations \eqref{case1_r} and \eqref{case1_b} reduce to the same equation, which coincides with the equation for the evolution of a single population of hard spheres as expected \cite{Bruna:2012cg}. 

In the following we define $\bar \alpha = \epsilon^d \alpha$, $\bar \gamma  = \epsilon^d \gamma$, and the total number density
\begin{equation} 
\label{rho_case1}
\rho ({\bf x},t) :  = r ({\bf x},t) + b({\bf x},t).
\end{equation}
When particles have the same size and diffusivity we find that 
\begin{equation} \label{relation_case1}
\rho \equiv 2 \phi/ \bar \gamma, 
\end{equation}
where $ \phi$ is the total volume density given  in \eqref{total_vol_density}. 
Using $\rho$, the equations \eqref{case1} can be rewritten in the following form
\begin{subequations}
\label{case1_rho} 
\begin{align}
\label{case1_rho_r}
\partial_t r   &=  \nabla  \cdot \left[ (1  -   \bar \gamma \rho  ) \nabla {  r}   +  (\bar \alpha + \bar \gamma) r  \nabla  \rho +  r  \nabla V_r +  \bar \gamma \nabla \left(  V_b -  V_r \right)   r b \right],   
\\
\label{case1_rho_b}
\partial_t b   &=  \nabla  \cdot \left[ ( 1-   \bar \gamma \rho  )  \nabla   b  + (\bar \alpha + \bar \gamma)   b   \nabla  \rho  +  b  \nabla V_b + \bar \gamma \nabla \left (  V_r -  V_b \right )  r  b \right ],
\end{align}
where we have used that $ \beta = \alpha + \gamma$.  
\end{subequations}
It is straight-forward to see that the system \eqref{case1_rho} has a formal gradient flow structure, with an entropy functional given by 
\begin{align}
\label{entropy_case1}
E(r, b)  = \int_{\Omega}    r\log   r +    b \log   b +  r  V_r +  b V_b +  \frac{\bar \alpha}{2} \left(    r^2 + 2   r  b  +    b^2 \right) d  {\bf x}.
\end{align}
Using the corresponding entropy variables
\begin{align} \label{entropy_vars_case1}
\begin{aligned}
u &:= \partial_r E= \log  r + \bar \alpha  \rho + V_r,\\
v &:= \partial_b E= \log  b + \bar \alpha  \rho + V_b,
\end{aligned}
\end{align}
the system can be written in the form
\begin{align}\label{case1_entropy}
\partial_t \begin{pmatrix}  r\\  b \end{pmatrix}  = 
\nabla \cdot \left[   M (r ,b)  \nabla 
\begin{pmatrix} u  \\ v \end{pmatrix}\right],
\end{align}
with the symmetric mobility matrix 
\begin{align}\label{mobility_case1}
\renewcommand{\arraystretch}{1.0}
M (r,b)  = 
\begin{pmatrix}    r (1 - \bar \gamma   b) & \bar \gamma   r   b  \\
\bar \gamma   r   b &        b (1 -  \bar \gamma   r)  \end{pmatrix}.
\end{align}

\subsection{Cross-diffusion system for particles of different size and diffusivity}

In this section we attempt to write a gradient flow structure for the general cross-diffusion system \eqref{pde_general} guided by the symmetric case in the previous subsection,  \eqref{entropy_case1} and \eqref{mobility_case1}. We will see that this requires a generalization of the definition of gradient flow structure. We define the following entropy 
\begin{subequations}
\label{gradflow_general}
\begin{align}
\label{entropy_general}
E_\epsilon (  r,  b)  =  \int_{\Omega}    r\log  r +   b \log  b +  r  V_r +  b  V_b +\frac{\alpha}{2} \left( \epsilon_r^d \,  r^2 + 2 \epsilon_{br}^d \,  r b  + \epsilon_b^d \,  b^2 \right) d  {\bf x},
\end{align}
and mobility matrix
\begin{align}
\label{mobility_general}
\renewcommand{\arraystretch}{1.0}
M_\epsilon ( r,  b) = 
\begin{pmatrix} D_r  r (1 - \gamma_r \epsilon_{br}^d  b) & D_r \gamma_b \epsilon_{br}^d  r  b  \\
D_b \gamma_r \epsilon_{br}^d  r  b &      D_b  b (1 -  \gamma_b \epsilon_{br}^d  r)  \end{pmatrix}.
\end{align}
\end{subequations}

As mentioned earlier, we suppose that the red and blue particle sizes are of the same asymptotic order, namely  $\epsilon_r \sim \epsilon_b$. It is then convenient to introduce a single small parameter $\epsilon$ and the order one parameters $a_r, a_b$ and $a_{br}$ such that $\epsilon_i^d = a_i \epsilon^d$. Then the entropy and mobility can be expressed as $E_\epsilon \sim E_0 + \epsilon^d E_1$ and $M_\epsilon \sim M_0 + \epsilon^d M_1$, with
\begin{align}
\label{grad_flowgeneral_exp}
\begin{aligned}
E_0 &= \int_{\Omega}  r\log  r +   b \log  b +  r V_r +  b V_b  \, d {\bf x}, & \  E_1 &=  \frac{\alpha}{2}  \int_{\Omega}  a_r   r^2 + 2 a_{br}   r b  + a_b   b^2  \, d  {\bf x} ,\\
M_0 &= \text{diag}( D_r r,  D_b b),&  M_1 &=  a_{br}  r  b \begin{pmatrix} -D_r \gamma_r & D_r \gamma_b  \\
D_b \gamma_r &   - D_b \gamma_b  \end{pmatrix}.
\end{aligned}
\end{align}

Using \eqref{gradflow_general}, the general cross-diffusion system \eqref{pde_general} can be rewritten as
\begin{align} \label{gradflow_generalasy}
\partial_t \begin{pmatrix}  r\\  b \end{pmatrix}  = 
\nabla \cdot \left[   M_\epsilon  \nabla 
\begin{pmatrix} \partial_r E_\epsilon  \\ \partial_b E_\epsilon \end{pmatrix} - \epsilon^{2d} G \right],
\end{align}
where $G(r,b)$ is the vector
\begin{align} \label{vectorw}
G = \alpha a_{br} r b  \begin{pmatrix}  \gamma_r (\theta_r \nabla r - \theta_b \nabla b) \\  \gamma_b(\theta_b \nabla b - \theta_r \nabla r)  \end{pmatrix},
\end{align}
with 
\begin{equation} 
\label{thetar_thetab}
\theta_r = D_b a_{br} - D_r a_r,\qquad \theta_b = D_r a_{br} - D_b a_b. 
\end{equation}
Then it is easy to see that the gradient flow structure induced by \eqref{gradflow_general} and our system \eqref{pde_general} (or \eqref{gradflow_generalasy}) agree up to order $\epsilon^d$, which is the order of the asymptotic expansion that produced \eqref{pde_general} in the first place. In other words, the discrepancy between the system \eqref{pde_general} and the gradient-flow induced by \eqref{gradflow_general} is of order $\epsilon^{2d}$. Therefore, up to order $\epsilon^d$,  we can see  \eqref{gradflow_generalasy} 
as a gradient flow structure of our system. We will call this an asymptotic gradient flow structure; the precise definition will be made clear in the following section. 

Finally, we note that the system \eqref{gradflow_general} coincides with the gradient-flow structure in the case $D_r = D_b$ and $\epsilon_r = \epsilon_b$,  see \eqref{entropy_case1} and \eqref{mobility_case1}. Note that $G \equiv 0$ for the parameter values of the simpler system \eqref{case1}, as expected. Specifically, we find that if $D_r = D_b$ and $\epsilon_r = \epsilon_b$, then $\theta_ r= \theta_b = 0$. A natural question to ask is whether there are other parameter values for which $G(r,b) \equiv 0$ for all $r, b$. Imposing that $\theta_ r= \theta_b = 0$ leads to the condition $a_{br}^2 = a_r ab$, which in turn leads to $\epsilon_r = \epsilon_b$, and thus that $D_r = D_b$. Therefore, the only case for which \eqref{gradflow_general} is an exact gradient flow for the system is the case which we have already studied, that is when the particle sizes and diffusivities are equal.

\section{Gradient Flows and Asymptotic Gradient Flows close to Equilibrium}
\label{sec:closetoequilibrium}

In the following we provide a more detailed discussion on gradient flow structures and implications for the behavior close to equilibrium.

\subsection{Full gradient flow structure case}
\label{sec:gradient_flow}

In this subsection we analyze the behavior of system \eqref{case1_entropy} close to equilibrium. 
We follow the strategy outlined in the previous subsection, by proving uniqueness of equilibrium solutions 
and studying the stability and well-posedness of the system close to this equilibrium solution. 


We have seen in the previous subsection that the linear stability analysis for gradient flow structures reduces to showing that the mobility matrix $M$ is positive definite in the case of a strictly convex entropy functional $E$, cf. \cite{Schlake:2011wr}. We assume from now on:
\begin{enumerate}[leftmargin=10mm,label=(\Alph*\Roman*)]
\item \label{a:V} ~~Let $V_r, V_b \in H^1(\Omega)\cap L^{\infty}(\Omega)$. 
\end{enumerate}
We recall that in case of assumption \ref{a:V} an equilibrium solution $(r_\infty,b_\infty)$ exists and that the corresponding entropy variables $u_\infty$ and $v_\infty$ are constant. The determinant of the mobility matrix $M$ defined in \eqref{mobility_case1} is given by
\begin{equation} \label{det_case1}
\det M = rb (1- \bar \gamma \rho).
\end{equation}
Together with the positivity of diagonal entries we see that $M$ is positive definite if $\rho < 1 / \bar \gamma$. This constraint gives a local bound on the total local volume density (using \eqref{relation_case1}), namely $2 \phi < 1$. This is consistent with the asymptotic assumption that $\phi \ll 1$.
Hence, we define the set 
\begin{equation}\label{equ:set}
\mathcal{S}=\left\{\begin{pmatrix}
r\\b
\end{pmatrix}\in \mathbb{R}^2:r\geq 0,b \geq 0,r+b \leq \frac{1}{\overline{\gamma}}\right\},
\end{equation}
which will also use in the existence proof presented in Section \ref{sec:existence}. For stability and uniqueness it will be crucial to have solutions staying strictly in the interior of $\mathcal{S}$, due to the degeneracy of the mobility matrix on the boundary of $\mathcal{S}$.

\begin{theorem}[Linear stability]\label{linearstability}
The stationary solutions of the  system \eqref{case1_entropy} are unique and linearly stable with respect to small perturbations $\xi, \eta \in L^2(0,T;H^1(\Omega))$.
\end{theorem}
\begin{proof}
Due to the gradient flow structure, any stationary solution of \eqref{case1_entropy} is a minimizer of the entropy subject to the constraints of given mass and $(r(x),b(x)) \in {\mathcal S}$ almost everywhere. Due to the strict convexity of the entropy and the convexity of the constraint set, the minimizer is unique.

Let us consider the linearisation of system \eqref{case1_entropy} around the unique equilibrium, which corresponds to the constant entropy variables $(u_\infty, v_\infty)$. As we have seen before this is equivalent to have a linear expansion in $(r,b)$ and in the entropy variables $(u,v)$ , i.e. $u=u_\infty+\xi, v=v_\infty+\eta$. In the latter setting we obtain the following first order approximation 

\begin{align*}
E^\star{''}(u_\infty,v_\infty) \begin{pmatrix} 
\partial_t \xi\\ \partial_t \eta
\end{pmatrix}=
\begin{pmatrix}
\partial_u r(u_\infty,v_\infty) \partial_t \xi+\partial_v r(u_\infty,v_\infty) \partial_t \eta \\
\partial_u b(u_\infty,v_\infty) \partial_t \xi+\partial_v b(u_\infty,v_\infty) \partial_t \eta 
\end{pmatrix}=\nabla \cdot \left( M(r_\infty,b_\infty)\begin{pmatrix}
\nabla \xi\\ \nabla \eta
\end{pmatrix} \right),
\end{align*}
where $E^\star{''}$ denotes the Hessian of the dual entropy functional. Note that for the first order approximation, we also have no flux boundary conditions. A simple calculation shows that $E^\star{''}(u_\infty,v_\infty)$ as well as $M(r,b)$ are positive definite for $(r,b)$ in the interior of ${S}$, which is guaranteed everywhere for the stationary solution $(r_\infty,b_\infty)$. 
Stability of this linear system is equivalent to nonpositivity of all the real parts of  eigenvalues $\lambda$ in 
$$ \lambda  E^\star{''}(u_\infty,v_\infty)  \begin{pmatrix}
 \xi\\ \eta
\end{pmatrix} = \nabla \cdot \left(  M(r_\infty,b_\infty)\begin{pmatrix}
\nabla \xi\\ \nabla \eta
\end{pmatrix} \right) .$$
Note that due to the symmetry  of the eigenvalue problem, all eigenvalues are real. 
Moreover, we find
\begin{align*}
\lambda \int_\Omega E^\star{''}(u_\infty,v_\infty)\begin{pmatrix}
 \xi\\ \eta
\end{pmatrix}\cdot \begin{pmatrix}
 \xi\\  \eta
\end{pmatrix}\, d{\bf x}=-\int_\Omega M(r_\infty,b_\infty)\begin{pmatrix}
\nabla \xi\\ \nabla \eta
\end{pmatrix}\cdot\begin{pmatrix}
\nabla \xi\\ \nabla \eta
\end{pmatrix}\,  d{\bf x}  .
\end{align*} 
Since $E^\star{''}$ and $M(r,b)$ are positive definite, we conclude that $\lambda<0$, which implies linear stability.
\end{proof}

Next we consider the well-posedness close to equilibrium. We shall make use of the following auxiliary lemma:
\begin{lemma}\label{lemma6}
Let $V_r$ and $V_b$ satisfy assumption \ref{a:V} and let $V_r,V_b \in X$ with
$$X=L^{\infty}(0,T;H^2(\Omega))\cap L^2(0,T;H^3(\Omega)) \cap H^1(0,T;H^1(\Omega)).$$
Then the gradient of the dual entropy functional $E^*{'}:X \times X \to X \times X, \,(u,v) \mapsto (r,b)$, defined by \eqref{entropy_vars_case1}, is continuous.
\end{lemma}
\begin{proof}
To verify continuity, we have to show the existence of a constant $C>0$ such that
\begin{equation}\label{equ:continuity}
\|(r,b)\|_{X\times X}\leq C \|(u,v)\|_{X\times X} \quad \forall (u,v)\in X\times X.\end{equation}
Given $(u,v)\in X\times X$, we calculate
\begin{align}\label{entropy_der1}
\begin{aligned}
\nabla u &=\frac{1}{r}\nabla r+\bar \alpha \nabla \rho +\nabla V_r, \qquad \nabla v =\frac{1}{b}\nabla b+\bar \alpha \nabla \rho +\nabla V_b,
\end{aligned}
\end{align}
\begin{align}\label{entropy_der2}
\begin{aligned}
\Delta u &=-\frac{1}{r^2} (\nabla r)^2 +\left(\frac{1}{r}+\bar \alpha\right)\Delta r +\bar \alpha \Delta b +\Delta V_r,\\
\Delta v &=-\frac{1}{b^2} (\nabla b)^2 +\left(\frac{1}{b}+\bar \alpha\right)\Delta b +\bar \alpha \Delta r +\Delta V_b,
\end{aligned}
\end{align}
and
\begin{align}\label{entropy_der3}
\begin{aligned}
\nabla \Delta u&=\frac{1}{r^3}\nabla r (\nabla r)^2-\frac{3}{r^2}\nabla r\Delta r+\left(\frac{1}{r}+\bar\alpha\right)\nabla \Delta r+\bar\alpha \nabla\Delta b+\nabla \Delta V_r,\\
\nabla \Delta v&=\frac{1}{b^3}\nabla b (\nabla b)^2-\frac{3}{b^2}\nabla b\Delta b+\left(\frac{1}{b}+\bar\alpha\right)\nabla \Delta b+\bar\alpha \nabla\Delta r+\nabla \Delta V_b.
\end{aligned}
\end{align}
From $H^2(\Omega)\hookrightarrow L^{\infty}(\Omega)$ in dimensions $d=2,3$ and using the definition of the entropy variables \eqref{entropy_vars_case1} we get that $r,b\in L^{\infty}(0,T;L^{\infty}(\Omega))$ and $r,b>\varepsilon$ for some positive $\varepsilon$. As $u,v\in L^{\infty}(0,T;H^2(\Omega))$ and $H^2(\Omega)\hookrightarrow W^{1,6}(\Omega)$, we get that $\nabla u, \nabla v \in L^{\infty}(0,T;L^6(\Omega))$ and therefore $\nabla u \nabla v \in L^{\infty}(0,T;L^3(\Omega))$.
Hence, relation \eqref{entropy_der1} implies that $\nabla r,\nabla b\in L^{\infty}(0,T;L^6(\Omega))$ and $\nabla r \nabla b \in L^{\infty}(0,T;L^3(\Omega))\hookrightarrow L^{\infty}(0,T;L^2(\Omega))$. Applying relation \eqref{entropy_der2}, we obtain that $\Delta r, \Delta b\in L^{\infty}(0,T;L^2(\Omega))$.
Since $u,v\in L^2(0,T;H^3(\Omega))$, the embedding $H^3(\Omega)\hookrightarrow W^{1,\infty}(\Omega)$ for dimensions $d=2,3$ as well as relation \eqref{entropy_der1} imply that $\nabla r,\nabla b\in L^2(0,T;L^\infty(\Omega))$. Together with relation \eqref{entropy_der3}, we obtain that $r,b \in L^2(0,T;H^3(\Omega))$, which implies continuity.
\end{proof}

\begin{theorem}[Well-posedness]\label{wellposedness}
Consider system \eqref{case1_entropy} with initial data $u_0,v_0 \in H^2(\Omega)$ and potentials $V_r,V_b\in H^3(\Omega)$. Furthermore let
\[\|u_0-u_\infty\|_{H^2(\Omega)}\leq \kappa \quad \text{and} \quad \|v_0-v_\infty\|_{H^2(\Omega)}\leq \kappa,\]
for $\kappa>0$ sufficiently small. Then, there exists a unique solution to system \eqref{case1_entropy} in 
\[B_R= \{(u,v):\|u-u_\infty\|_X\leq R,\, \|v-v_\infty\|_X\leq R\},\]
where 
\[ X=L^{\infty}(0,T;H^2(\Omega))\cap L^2(0,T;H^3(\Omega)) \cap H^1(0,T;H^1(\Omega))\] 
and $R$ is a constant depending on $\kappa$ and $T>0$ only. 
\end{theorem}

\begin{proof}
The proof is based on Banach's fixed point theorem. The corresponding fixed point operator is defined by the evolution of $u-u_\infty$ and $v-v_\infty$, which can be written as
\begin{align}\label{equ1} 
\begin{aligned}
E^\star{''}(u_\infty,v_\infty)\begin{pmatrix}
\partial_t (u-u_\infty) \\ \partial_t (v-v_\infty)
\end{pmatrix}&-\nabla \cdot \left( M(r_\infty,b_\infty)\begin{pmatrix}
\nabla u \\ \nabla v \end{pmatrix}\right)\\
&=\nabla \cdot \Bigl(( M(r,b)-M(r_\infty,b_\infty))
\begin{pmatrix}
\nabla(u-u_\infty)\\ \nabla(v-v_\infty)
\end{pmatrix} \Bigr)\\
&\quad-(E^\star{''}(u,v)-E^\star{''}(u_\infty,v_\infty))\begin{pmatrix}
\partial_t (u-u_\infty) \\ \partial_t (v-v_\infty)
\end{pmatrix}\\
&=:F(u,v),
\end{aligned}
\end{align}
where we used that $(r,b)=E^\star{'}(u,v)$. Note that by using a similar argumentation as in the proof of Lemma \ref{lemma6}, we can show that the stationary solutions $r_\infty,b_\infty$ are in $H^3(\Omega)$ assuming that the potentials $V_r,V_b$ are in $H^3(\Omega)$. Consider $(u,v)\in X\times X$ with the corresponding function $r=r(u,v),b=b(u,v)$ and let $L$ denote the solution of \eqref{equ1} for a given right-hand side. Then the fixed point operator is given by the concatenation of $L$ and $F$, that is
\[J=L\circ F:X \times X \to X\times X.\]
Note that Lemma \ref{lemma6} guarantees that $(r,b)=(r(u,v),b(u,v))\in X\times X$. Properties of the entropy functional guarantee that $E^\star{''}$ is bounded for $(u,v)\in X\times X$.
The operator $F$ defined in \eqref{equ1} maps from $X\times X$ into $Y\times Y$, where
\[Y:=L^{\infty}(0,T;L^2(\Omega))\cap L^2(0,T;H^1(\Omega)).\] 
Standard results for linear parabolic equations, see \cite{ol1968linear} or \cite{Evans199806}, ensure that the solution $(\tilde{u}-u_\infty,\tilde{v}-v_\infty)$ to equation \eqref{equ1} lie in $X\times X$.

To apply Banach's fixed point theorem, it remains to show that the operator $J$ is self-mapping into the ball $B_R$ and contractive.
The selfmapping property follows from the fact that 
\begin{align*}
\|(\tilde{u}-u_\infty,\tilde{v}-v_\infty)\|_{X\times X}&\leq C \Bigl(\underbrace{\|F(u,v)\|_{L^2}}_{\sim R^2}+\underbrace{\|(u_0-u_\infty,v_0-v_\infty)\|_{H_0^1}}_{\sim\kappa}\Bigr)=:R(\kappa).
\end{align*}
For the contractivity we consider $(u_1,v_1)\in X\times X$ and $(u_2,v_2)\in X \times X$ and deduce that:
\begin{align*}
\|F(u_1,v_1)-F(u_2,v_2)\|_Y&=\left\|\nabla \cdot \left( \left(M(E^*{'}(u_1,v_1))-M(E^*{'}(u_\infty,v_\infty))\right)\begin{pmatrix}
\nabla(u_1-u_\infty)\\ \nabla(v_1-v_\infty)
\end{pmatrix} \right)\right.\\
&\quad +(E^\star{''}(u_1,v_1)-E^\star{''}(u_\infty,v_\infty))\begin{pmatrix}
\partial_t (u_1-u_\infty) \\ \partial_t (v_1-v_\infty)
\end{pmatrix}\\
&\quad -\nabla \cdot \left(  \left(M(E^*{'}(u_2,v_2))-M(E^*{'}(u_\infty,v_\infty))\right)\begin{pmatrix}
\nabla(u_2-u_\infty)\\ \nabla(v_2-v_\infty)
\end{pmatrix} \right)\\
& \quad \left.-(E^\star{''}(u_2,v_2)-E^\star{''}(u_\infty,v_\infty))\begin{pmatrix}
\partial_t (u_2-u_\infty) \\ \partial_t (v_2-v_\infty)
\end{pmatrix}\right\|_Y
\end{align*}
Therefore 
\begin{align*}
\|F(u_1,v_1)-F(u_2,v_2)\|_Y &\leq \phantom{+} \left\|\nabla \cdot \left( \left(M(E^*{'}(u_1,v_1))-M(E^{*}{'} (u_2,v_2))\right)\begin{pmatrix}
\nabla(u_1-u_\infty)\\ \nabla(v_1-v_\infty)
\end{pmatrix} \right)\right\|_Y\\
&\quad+\left\|\nabla \cdot \left(  \left(M(E^*{'}(u_2,v_2))-M(E^*{'}(u_\infty,v_\infty))\right)\begin{pmatrix}
\nabla(u_1-u_2)\\ \nabla(v_1-v_2)
\end{pmatrix} \right)\right\|_Y\\
&\quad+\left\|(E^\star{''}(u_1,v_1)-E^\star{''}(u_2,v_2))\begin{pmatrix}
\partial_t (u_1-u_\infty) \\ \partial_t (v_1-v_\infty)
\end{pmatrix}\right\|_Y\\
&\quad+\left\|(E^\star{''}(u_2,v_2)-E^\star{''}(u_\infty,v_\infty))\begin{pmatrix}
\partial_t (u_1-u_2) \\ \partial_t (v_1-v_2)
\end{pmatrix}\right\|_Y\\
&\leq C_1R (\|u_1-u_2\|_X+\|v_1-v_2\|_X),
\end{align*}
for some constant $C_1>0$.
Hence, we have that
\begin{align*}
\|J(u_1,v_1)-J(u_2,v_2)\|_X\leq CR(\|u_1-v_1\|_X+\|u_2-v_2\|_X),
\end{align*}
for some $C>0$.
Choosing $\kappa$ and $R$ such that $R<\frac{1}{C}$, we can apply Banach's fixed point theorem which guarantees the existence of unique solutions $(u,v) \in B_R$.
\end{proof}

\subsection{Asymptotic Gradient Flow Structure}
\label{sec:asymptoticgradientflowstructure}

We have seen in Section \ref{sec:case1} that system \eqref{pde_general} with particles of same size satisfies a gradient flow structure, which is not valid for the general system due to terms of higher order in $\epsilon$. However, we want to interpret the latter as an {\em asymptotic gradient flow structure}, motivated by the fact that it was derived from an asymptotic expansion in $\epsilon$. For further motivation, consider a gradient flow structure for the density $w$ of the form
\begin{equation} \label{eq:wequation}
 \partial_t w = \nabla \cdot (M(w;\delta) \nabla E'(w;\delta)), 
\end{equation}
where both the mobility $M$ and the entropy $E$ depend on a small parameter $\delta > 0$. With an expansion of $M$ and $E$ in terms of $\delta$ as
$$ M(w;\delta) = \sum_{j=0}^\infty \delta^j M_j(w), \text{ and } E(w;\delta) = \sum_{j=0}^\infty \delta^j E_j(w), $$
we find 
$$ \partial_t w  = \sum_{k=0}^\infty \delta^k  \nabla \cdot \left( \sum_{j=0}^k M_j(w) \nabla E_{k-j}'(w) \right). $$
Truncating the expansion on the right-hand side at a finite $k$ does not yield a gradient flow structure in general, but up to terms of order $\delta^{k}$ it coincides with the gradient flow structure with mobility
$ \sum_{j=0}^k \delta^j M_j(w)$ and entropy $ \sum_{j=0}^k \delta^j E_j(w)$. In our case we deal with the example $k=1$ (with $\delta = \epsilon^d$), where we have
$$ \partial_t w =  \nabla \cdot (M_0(w) \nabla E_0'(w)) + \delta  \nabla \cdot (M_1(w) \nabla E_0'(w)+M_0(w) \nabla E_1'(w)).$$
Adding a term of order $\delta^2$, namely  $\delta^2  \nabla \cdot (M_1(w) \nabla E_1'(w))$, 
this equation becomes a gradient flow. This motivates a more general definition:

\begin{definition}
Let ${\mathcal F(.;\delta)}$ be a densely defined operator on some Hilbert space for $\delta \in (0,\delta_*)$. Then the dynamical system
\begin{equation} \label{eq:dynsyst}
\partial_t w = {\mathcal F}(w;\delta)
\end{equation}
is called an asymptotic gradient flow structure of order $k$ if there exist densely defined 
operators ${\mathcal G}_j$, $j={k+1},\ldots,2k$ such that for $\delta \in (0,\delta_*)$
\begin{equation}
{\mathcal F}(w;\delta) + \sum_{j=k+1}^{2k} \delta^j {\mathcal G}_{k+1+j}(w) = - {\mathcal M}(w;\delta) {\mathcal E}'(w;\delta)
\end{equation}
for some (parametric) energy functional ${\mathcal E}(\cdot;\delta)$, and ${\mathcal M}(w;\delta)$
is a densely  defined formally positive-definite operator for each $w$.
\end{definition} 

If an expansion of mobility and entropy up to order $k$ are available, it seems natural to perform a separate expansion to derive a lower order model that is a gradient flow as well. For complicated models and types of expansions as in \cite{Bruna:2012wu} or \cite{Bruna:2012cg} it seems not suitable to derive such however. Hence, we shall work with the asymptotic gradient flow concept below. Note that with the above notations we can rewrite \eqref{eq:dynsyst} as
\begin{equation}
\partial_t w   =  - {\mathcal M}(w;\delta) {\mathcal E}'(w;\delta)- \delta^{k+1} \sum_{j=0}^{k-1} \delta^j {\mathcal G}_{k+1+j}(w), 
\end{equation}
which opens the door to perturbation arguments in the analysis of \eqref{eq:dynsyst} for $\delta$ sufficiently small. 

In the remainder of this section we will highlight in particular the use of asymptotic gradient flow structures close to equilibrium. Let $w_\infty^{\delta}$ denote the  equilibrium solution, which is a minimizer of the energy functional on the manifold defined by ${\mathcal M}$. Hence $w_{\infty}^\delta$ solves ${\mathcal M}(w;\delta) {\mathcal E}'(w_\infty^\delta;\delta) = 0$ for any $w$. In the case of \eqref{eq:wequation} it typically means that $E'(w_\infty^\delta;\delta)$ is constant. In order to prove the existence of a stationary solution of \eqref{eq:dynsyst} one can then try the following strategy: first of all compute $w_\infty^\delta$ (or prove at least its existence and uniqueness by variational principles) and then use the equation 
$$
{\mathcal M}(w_\infty^\delta;\delta) {\mathcal E}'(w;\delta) = -\delta^{k+1} \sum_{j=0}^{k-1} \delta^j {\mathcal G}_{k+1+j}(w) + ({\mathcal M}(w_\infty^\delta;\delta) - {\mathcal M}(w ;\delta)) ( {\mathcal E}'(w;\delta) -  {\mathcal E}'(w_\infty^\delta;\delta) ) $$ 
as the basis of a fixed-point argument, freezing $w$ on the right-hand side. Since the terms on the right-hand side are of high order in $\delta$ or of second order in terms of $w-w_\infty^\delta$, there is some hope of contractivity of the fixed-point operator close to equilibrium $w_\infty^\delta$. Such an approach can also yield some structural insight into the stationary solution, since it will be a higher order perturbation of $w_\infty^\delta$. The same idea can be employed to analyze transient solutions of \eqref{eq:dynsyst}, since
\begin{align*}
\partial_t w + {\mathcal M}(w_\infty^\delta;\delta) {\mathcal E}'(w;\delta) = &-\delta^{k+1} \sum_{j=0}^{k-1} \delta^j {\mathcal G}_{k+1+j}(w) \\
&+ ({\mathcal M}(w_\infty^\delta;\delta) - {\mathcal M}(w ;\delta)) ( {\mathcal E}'(w;\delta) -  {\mathcal E}'(w_\infty^\delta;\delta) ) .
\end{align*}
If $ {\mathcal M}(w_\infty^\delta;\delta) $ is invertible and ${\mathcal E}(\cdot;\delta)$ is strictly convex on its domain, one can directly apply variational techniques to analyze the fixed point operator. In particular it can be rather beneficial to set up the fixed-point operator in dual (or entropy) variables $z = {\mathcal E}'(w;\delta)$ instead. 

Finally let us comment on the linear stability analysis around a stationary solution $w_*^\delta$. Using a similar way of expanding the equation around $w_\infty^\delta$, the linearised problem for a variable $\tilde w$ is given by
\begin{eqnarray*} \partial_t \tilde w + {\mathcal M}(w_\infty^\delta;\delta) ({\mathcal E}''(w_*^\delta;\delta) \tilde w) &=& -\delta^{k+1} \sum_{j=0}^{k-1} \delta^j {\mathcal G}_{k+1+j}'(w_*^\delta)\tilde w + \\&& ({\mathcal M}(w_\infty^\delta;\delta) - {\mathcal M}(w_*^\delta;\delta)) ( {\mathcal E}''(w_*^\delta;\delta)\tilde w) - \\  && 
  ( {\mathcal M}'(w_*^\delta;\delta)\tilde w) ( {\mathcal E}'(w_*^\delta;\delta) -  {\mathcal E}'(w_\infty^\delta;\delta) ), \end{eqnarray*}
 where we denote by ${\mathcal E}'$ and ${\mathcal M}'$ the derivatives with respect to $w$ at fixed $\delta$.
Due to positive definiteness of ${\mathcal E}''(w_*^\delta;\delta)$, this system can be interpreted as a linear equation for the linearised entropy variable $\tilde z = {\mathcal E}''(w_*^\delta;\delta) \tilde w$, which is equivalent to considering linear stability directly in the transformed equation for the entropy variable $z$ as performed in \cite{Schlake:2011wr}. Using the simplified notation ${\mathcal A}={\mathcal E}''(w_*^\delta;\delta)^{-1}$ and ${\mathcal B}={\mathcal M}(w_\infty^\delta;\delta)$, we obtain 
\begin{eqnarray*} {\mathcal A} \partial_t \tilde z + {\mathcal B}\tilde z &=& -\delta^{k+1} \sum_{j=0}^{k-1} \delta^j {\mathcal G}_{k+1+j}'(w_*^\delta){\mathcal A} \tilde z + ({\mathcal B} - {\mathcal M}(w_*^\delta;\delta)) \tilde z - \\  && 
  - ( {\mathcal M}'(w_*^\delta;\delta){\mathcal A} \tilde z) ( {\mathcal E}'(w_*^\delta;\delta) -  {\mathcal E}'(w_\infty^\delta;\delta) ). \end{eqnarray*}
In the case of a gradient flow (${\mathcal G}_j \equiv 0$, $w_*^\delta = w_\infty^\delta$) this reduces to 
$$  {\mathcal A} \partial_t \tilde z + {\mathcal B}\tilde z = 0, $$
which is stable if ${\mathcal A}$ and ${\mathcal B}$ are positive definite.
In the asymptotic gradient flow case, with $w_*^\delta = w_\infty^\delta + \mathcal O(\delta^{k+1})$, we can formally write the linearised problem as
\begin{equation}
{\mathcal A} \partial_t \tilde z + ({\mathcal B}+ \delta^{k+1} {\mathcal C})\tilde z = 0, 
\end{equation}
and hence expect linear stability also for $w_*^\delta$ if $\delta$ is sufficiently small. 

The application of the above strategies to prove existence of solutions and linear stability to a concrete model obviously depends on an appropriate choice of topologies. In the remaining part of this section we focus on the analysis of the asymptotic gradient flow of the general model.

\subsection{Asymptotic gradient flow structure case}
\label{sec:asymptotic_gradient_flow}


First we study the existence of stationary solutions to \eqref{gradflow_generalasy}. Then we discuss stability of stationary states following the ideas presented in subsection \ref{sec:asymptoticgradientflowstructure}.

Note that for $\epsilon=0$, the equilibrium solutions are given by $(r_\infty,b_\infty)=(C_r e^{-V_r},C_b e^{-V_b})$, with constants $C_r$ and  $C_b$ depending on the initial masses only. Hence $(r_{\infty}, b_{\infty})$ are bounded for $V_r$ and $V_b$ satisfying assumption \ref{a:V}. For $\epsilon>0$, the equilibrium solutions are a $\mathcal{O}(\epsilon^d)$ perturbation in $L^{\infty}$ and therefore also uniformly bounded. 
\begin{theorem}[Existence of stationary solutions]\label{theorem2} 
Consider system \eqref{gradflow_generalasy} with potentials $V_r,V_b\in H^3(\Omega)$. Then there exists a unique stationary state $(u_*, v_*)$ to system \eqref{gradflow_generalasy} in 
\[B_R= \{(u,v):\|u-u_\infty\|_X \leq R,\, \|v-v_\infty\|_X \leq R\},\]
where $X=H^3(\Omega)$ and $R$ depending on $\epsilon$ and $T>0$ only. 
\end{theorem}

\begin{proof}
We follow the ideas detailed in Subsection \ref{sec:asymptoticgradientflowstructure} and define a fixed point operator close to equilibrium. Denote by $(r_\infty,b_\infty)$ the minimizer of the entropy functional $E_\epsilon(r,b)$, which exists as the entropy functional is strictly convex. Then any stationary solution to system \eqref{gradflow_generalasy} exists has to satisfy
\begin{align}\label{equ2}
-\nabla \cdot \left( M(r_\infty,b_\infty)
\begin{pmatrix}
\nabla u \\ \nabla v 
\end{pmatrix}\right)
&=\nabla \cdot \left(-\epsilon^{2d} G(r,b)+ \left(M(r,b)-M(r_\infty,b_\infty)\right)\begin{pmatrix} \nabla(u-u_\infty)\\ \nabla(v-v_\infty)
\end{pmatrix} \right)\nonumber\\
&=:F(u,v).
\end{align}
Similar arguments as in Lemma \ref{lemma6} ensure that for $(u,v)\in X\times X$ the functions $r=r(u,v)$ and $b=b(u,u)$ lie in $X\times X$.
Let $L$ denote the solution operator to \eqref{equ2} for a given right-hand side $F(u,v)$. Then the fixed point operator is constructed by:
\[J=L\circ F:X \times X \to X\times X.\]
Hence, we can conclude that $F$ maps from $X\times X$ into $Y\times Y$, where $Y=H^1(\Omega)$.
Employing results about the elliptic operator, cf. \cite{gilbarg2015elliptic} or \cite{Evans199806}, we obtain that the solution $(\tilde{u},\tilde{v})$ to equation \eqref{equ2} is in $X\times X$.

To apply Banach's fixed point theorem, it remains to show that the operator $J$ is self-mapping into the ball $B_R$ and contractive. The self-mapping property follows from the fact that 
\begin{align*}
\|(\tilde{u},\tilde{v})\|_{X\times X}&\leq \tilde{C} \underbrace{\|F(u,v)\|_{L^2}}_{\sim R^2+\epsilon^{2d}}=:R(\epsilon).
\end{align*}
For the contractivity we consider $(u_1,v_1)\in X\times X$ and $(u_2,v_2)\in X \times X$. Then
\begin{align*}
\|F(u_1,v_1)&-F(u_2,v_2)\|_Y=\left\|\nabla \cdot \left( -\epsilon^{2d} G(E^*{'}(u_1,v_1))d +\epsilon^{2d} G(E^*{'}(u_2,v_2))\right)\right.\\
&\qquad+\nabla \cdot \left( \left(M(E^*{'}(u_1,v_1))-M(E^*{'}(u_\infty,v_\infty))\right)\begin{pmatrix}
\nabla(u_1-u_\infty)\\ \nabla(v_1-v_\infty)
\end{pmatrix} \right)\\
&\left.\qquad -\nabla \cdot \left( \left(M(E^*{'}(u_2,v_2))-M(E^*{'}(u_\infty,v_\infty))\right)\begin{pmatrix}
\nabla(u_2-u_\infty)\\ \nabla(v_2-v_\infty)
\end{pmatrix} \right)\right\|_Y.
\end{align*}
Therefore 
\begin{align*}
\|F(u_1,v_1)-F(u_2,v_2)\|_Y &\leq \left\|\nabla \cdot \left(\epsilon^{2d} G(E^*{'}(u_1,v_1))-\epsilon^{2d} G(E^*{'}(u_2,v_2)\right)\right\|_Y\\
&\quad+\left\|\nabla \cdot \left( \left(M(E^*{'}(u_1,v_1))-M(E^{*}{'} (u_2,v_2))\right)\begin{pmatrix}
\nabla(u_1-u_\infty)\\ \nabla(v_1-v_\infty)
\end{pmatrix} \right)\right\|_Y\\
&\quad+\left\|\nabla \cdot \left(  \left(M(E^*{'}(u_2,v_2))-M(E^*{'}(u_\infty,v_\infty))\right)\begin{pmatrix}
\nabla(u_1-u_2)\\ \nabla(v_1-v_2)
\end{pmatrix} \right)\right\|_Y\\
&\leq\epsilon^{2d} C_1 \left(\left\| r_1-r_2\right\|_X +\left\|b_1-b_2\right\|_X\right)+ C_2R (\|u_1-u_2\|_X+\|v_1-v_2\|_X)\\
&\leq \left(\epsilon^{2d} C_3+2C_1R\right)(\|u_1-u_2\|_X+\|v_1-v_2\|_X),
\end{align*}
for some constants $C_1,C_2,C_3>0$ and therefore
\begin{align*}
\|J(u_1,v_1)-J(u_2,v_2)\|_X\leq \tilde{C}  \left(\epsilon^{2d} C_3+2C_1R\right) (\|u_1-v_1\|_X+\|u_2-v_2\|_X),
\end{align*}
for some $C>0$. Choosing $R$ and $\epsilon$ such that $$\tilde{C} \left(\epsilon^{2d}C_3 +2C_1R\right)<1,$$ we can apply Banach's fixed point theorem which guarantees the existence of unique solutions $(u_*,v_*) \in B_R$.
\end{proof}

A direct consequence of the proof is the closeness of the stationary solution $(u_*,v_*)$ to the gradient flow solution $(u_\infty,v_\infty)$:

\begin{corollary} \label{corollary_closeness}
Let the assumptions of Theorem \ref{theorem2} be satisfied. Then there exists a constant $C > 0$ such that for $\epsilon$ sufficiently small
\begin{equation}
\Vert u_* - u_\infty \Vert_X + \Vert v_* - v_\infty \Vert_X \leq C \epsilon^{2d}.
\end{equation}
\end{corollary}
\begin{proof}
We use \eqref{equ2} rewritten as 
\begin{align*}
&-\nabla \cdot \left( M(r_\infty,b_\infty)\begin{pmatrix}
\nabla (u_* - u_\infty) \\ \nabla (v_*-v_\infty) \end{pmatrix}\right) = \\ & \qquad \qquad \qquad \qquad {\nabla \cdot \left(-\epsilon^{2d} G(r,b)+ \left(M(r_*,b_*)-M(r_\infty,b_\infty)\right)\begin{pmatrix}
\nabla(u_*-u_\infty)\\ \nabla(v_*-v_\infty)
\end{pmatrix} \right)}.
\end{align*}
and the properties of the operators used above immediately imply the assertion.
\end{proof}


We conclude this section by discussing linear stability of system \eqref{gradflow_generalasy} close to its stationary states $(u_*, v_*)$. Following the ideas presented in Section \ref{sec:asymptoticgradientflowstructure} we rewrite \eqref{gradflow_generalasy} as
\begin{align*}
\partial_t (r,b) &=\mathcal{M}(r,b)\mathcal{E}'(r,b)-\epsilon^{2d}\mathcal{G}(r,b).
\end{align*}
Then
\begin{align}\label{stationary_stability1}
\begin{split}
\partial_t (r,b) -\mathcal{M}(r_\infty,b_\infty)\mathcal{E}'(r,b)=&-\epsilon^{2d}\mathcal{G}(r,b)\\
&+(\mathcal{M}(r,b)-\mathcal{M}(r_\infty,b_\infty))(\mathcal{E}'(r,b)-\mathcal{E}'(r_\infty,b_\infty)).
\end{split}
\end{align}
The linearisation of equation \eqref{stationary_stability1} around $(r_*,b_*)$ is given by the following system for  $(\tilde{r},\tilde{b})$:
\begin{align*}
\begin{aligned}
\partial_t (\tilde{r},\tilde{b})-\mathcal{M}(r_\infty,b_\infty)(\mathcal{E}{''}(r_*,b_*)(\tilde{r},\tilde{b}))&=-\epsilon^{2d}\mathcal{G}'(r_*,b_*)(\tilde{r},\tilde{b})\\
&\quad+(\mathcal{M}(r_*,b_*)-\mathcal{M}(r_\infty,b_\infty))(\mathcal{E}{''}(r_*,b_*)(\tilde{r},\tilde{b}))\\
&\quad+(\mathcal{M}'(r_*,b_*) (\tilde{r},\tilde{b}))(\mathcal{E}'(r_*,b_*)-\mathcal{E}'(r_\infty,b_\infty)).
\end{aligned}
\end{align*}
Using the linearised entropy variables $(\tilde{u},\tilde{v})=\mathcal{E}{''}(r_*,b_*)(\tilde{r},\tilde{b})$ we obtain\(\)
\begin{align}\label{eq:linstab} 
\begin{aligned}
\mathcal{A} \partial_t (\tilde{u},\tilde{v})-\mathcal{B}(\tilde{u},\tilde{v})&=-\epsilon^{2d}\mathcal{G}'(r_*,b_*)\mathcal{A}(\tilde{u},\tilde{v})+(\mathcal{M}(r_*,b_*)-\mathcal{B})(\tilde{u},\tilde{v})\\
&\quad+(\mathcal{M}'(r_*,b_*) \mathcal{A}(\tilde{u},\tilde{v}))(\mathcal{E}'(r_*,b_*)-\mathcal{E}'(r_\infty,b_\infty)),
\end{aligned}
\end{align}
where $\mathcal{A}= \mathcal{E}{''}^{-1}(r_*,b_*)$ is a positive and $\mathcal{B}=\mathcal{M}(r_\infty,b_\infty)$ are negative semidefinite operator. Note that with the
usual settings for elliptic systems, $\mathcal{B}$ is elliptic and hence invertible on the space of function pairs in $H^1(\Omega)$ with zero means.

As already mentioned in Section \ref{sec:asymptoticgradientflowstructure}, $(r_*,b_*)=(r_\infty,b_\infty) + \mathcal O(\epsilon^{2d})$ and \eqref{eq:linstab} can be written as
\begin{align*} 
\begin{aligned}
\mathcal{A} \partial_t (\tilde{u},\tilde{v})-(\mathcal{B}+\epsilon^{2d}\mathcal{C})(\tilde{u},\tilde{v})&=0,
\end{aligned}
\end{align*}
for some bounded operator $\mathcal{C}$ on $H^1(\Omega)^2$ .
As $\mathcal{B}$ is symmetric and negative definite except on the two-dimensional space of constant functions also annihilated by $\mathcal{C}$, the nonzero eigenvalues of $\mathcal{B}+\epsilon^{2d}\mathcal{C}$ stay negative for $\epsilon$ sufficiently small, yielding linear stability for $(r_*,b_*)$, cf. \cite{kato2013perturbation}.

\section{Numerical investigations of steady states}  \label{sec:numerics}

In this section we compute the stationary solutions of \eqref{pde_general}. For the symmetric system \eqref{case1}, the solutions can be computed exactly as the  minimizers of the entropy $E$ in \eqref{entropy_case1}. If the mobility matrix \eqref{mobility_case1} is positive definite (which it is under the assumptions),  the equilibrium states can be computed by finding constants $\chi_r \in \mathbb{R}$ and $\chi_b \in \mathbb{R}$ such that
\begin{align*}
\partial_r E = \chi_r \text{ and } \partial_b E = \chi_b
\end{align*}
subject to normalization constraints. In the case of system \eqref{case1} we have
\begin{subequations}\label{stationary}
\begin{align}
\log r_{\infty} + V_r + \alpha(\epsilon_r^d r_{\infty} + \epsilon_{br}^d b_{\infty}) &= \chi_b\\
\log b_{\infty} + V_b + \alpha(\epsilon_b^d b_{\infty} + \epsilon_{br}^d r_{\infty}) &= \chi_r\\
\int_{\Omega} r_{\infty}( {\bf x}) \, d {\bf x} &= N_r\\
\int_{\Omega} b_{\infty}( {\bf x}) \, d {\bf x} &= N_b.
\end{align}
\end{subequations}
System \eqref{stationary} defines a nonlinear operator equation $F(r_{\infty}, b_{\infty}, \chi_r, \chi_b) = 0$, which can be solved via Newton's method. Note that the no-flux boundary conditions are automatically satisfied by assuming that $\partial_r E$ and $\partial_b E$ are constant.

For the general case \eqref{pde_general} we only obtain an asymptotic gradient flow structure with the entropy  $E_\epsilon$; if we use \eqref{stationary} to solve for the stationary solutions we will be committing an order $\epsilon^{2d}$ error. Instead, we compute the exact stationary states $(r_*, b_*)$ of the general system by solving the time-dependent problem \eqref{pde_general} for long-times, until the system has equilibrated. To solve \eqref{pde_general}, we use a second-order accurate finite-difference scheme in space and the method of lines with the inbuilt Matlab ode solver \texttt{ode15s} in time. 

We set $d=2$ and consider one-dimensional external potentials $\tilde V_r = \tilde V_r(x)$ and $\tilde V_b = \tilde V_b(x)$ so that the stationary states will be also one-dimensional. In particular, we take linear potentials $\tilde V_r = v_r x$ and $\tilde V_b = v_b x$ and solve for the full system \eqref{pde_general} and for the minimizers \eqref{stationary} in $[-1/2,1/2]$, which is split into 200 intervals. The Newton solver is initialized with the stationary state solution in the case of point particles and terminated if $\lVert F(r,b,\chi_r, \chi_b) \rVert_{L^2(0,1)} \leq 10^{-8}$.

\paragraph{Example 1} First we consider the case: $\epsilon_r = \epsilon_b$ and $D_r = D_b$, that is particles of the same size and diffusivity. In this case, system \eqref{case1} has a full gradient flow structure and hence we expect that the stationary states computed with the two approaches to be the same.
We plot the two pairs, $(r_*, b_*)$ computed as the long-time limit of \eqref{case1}, and $(r_\infty, b_\infty)$, computed from \eqref{stationary} in \figref{fig:stat_exact}. The parameters are $D_r = D_b = 1$, $\epsilon_r = \epsilon_b = 0.01$, $N_b = N_r = 200$ and $v_r = 2$, $v_b = 1$. As expected, the solutions are identical.
\def \scc {0.7}
\def \scl {1.1}
\begin{figure}[ht]
\unitlength=1cm
\begin{center}
\psfragscanon
\psfrag{x}[][][\scl]{$x$} \psfrag{data1}[][][\scl]{$r_\infty$}  \psfrag{data2}[][][\scl]{$b_\infty$}  \psfrag{data3}[][][\scl]{$r_*$} \psfrag{data4}[][][\scl]{$b_*$}
\includegraphics[width = .7\linewidth]{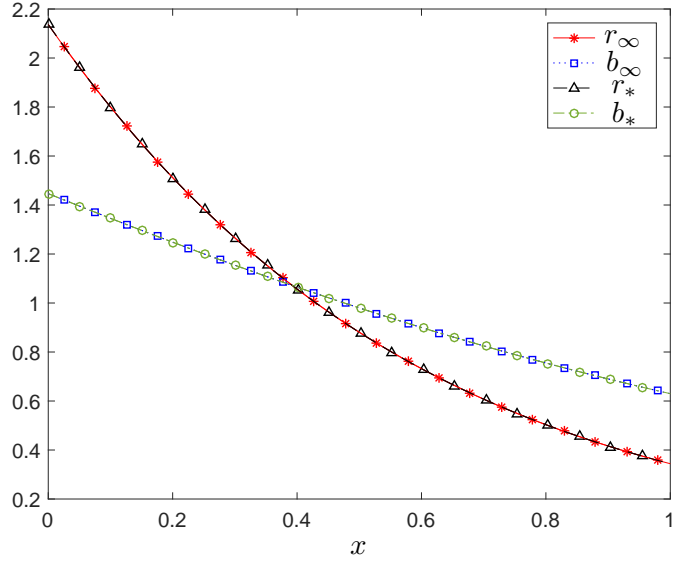}
\caption{Stationary solutions $(r_*, b_*)$ and $(r_\infty, b_\infty)$ from solving the long-time limit of \eqref{pde_general} and \eqref{stationary}, respectively, in the case with $\theta_r = \theta_b = 0$. The parameter values are $d=2$, $D_r = D_b = 1$, $\epsilon_r = \epsilon_b = 0.01$, $N_b = N_r = 200$ and $\tilde V_r = 2x$, $\tilde V_b = x$.}
\label{fig:stat_exact}
\end{center}
\end{figure}

\paragraph{Example 2}
From Corollary \ref{corollary_closeness} we expect the stationary solutions corresponding to the case of an asymptotic and a full gradient flow equation agree up to order $\mathcal{O}(\epsilon^d)$.
To investigate this, we again compare the solutions $(r_*, b_*)$ and $(r_\infty, b_\infty)$ as we move away from the case with an exact gradient-flow structure (which corresponds to $\theta_r = \theta_b = 0$, see \eqref{gradflow_generalasy} and \eqref{thetar_thetab}). 

In particular, we do a one-parameter sweep with $\theta_r$, increasing it from 0 (as in \figref{fig:stat_exact}) to $9 \cdot 10^{-5}$, while keeping $\epsilon_r = \epsilon_b =0.01$ and  $D_b = 1$ fixed. This ensures that when $\theta_r = 0$ then $\theta_b = 0$. The reds diffusivity $D_r$ is varied according to \eqref{thetar_thetab}. We plot the result for $\theta_r = 8\cdot 10^{-5}$ in \figref{fig:stat_error}. As expected, the error between the stationary solutions is apparent. 
\def \scc {0.8}
\def \scl {1.1}
\begin{figure}[ht]
\unitlength=1cm
\begin{center}
\psfragscanon
\psfrag{x}[][][\scl]{$x$} \psfrag{xx}[][][\scc]{$x$} \psfrag{data1}[][][\scl]{$r_\infty$}  \psfrag{data2}[][][\scl]{$b_\infty$}  \psfrag{data3}[][][\scl]{$r_*$} \psfrag{data4}[][][\scl]{$b_*$}
\includegraphics[width = .7\linewidth]{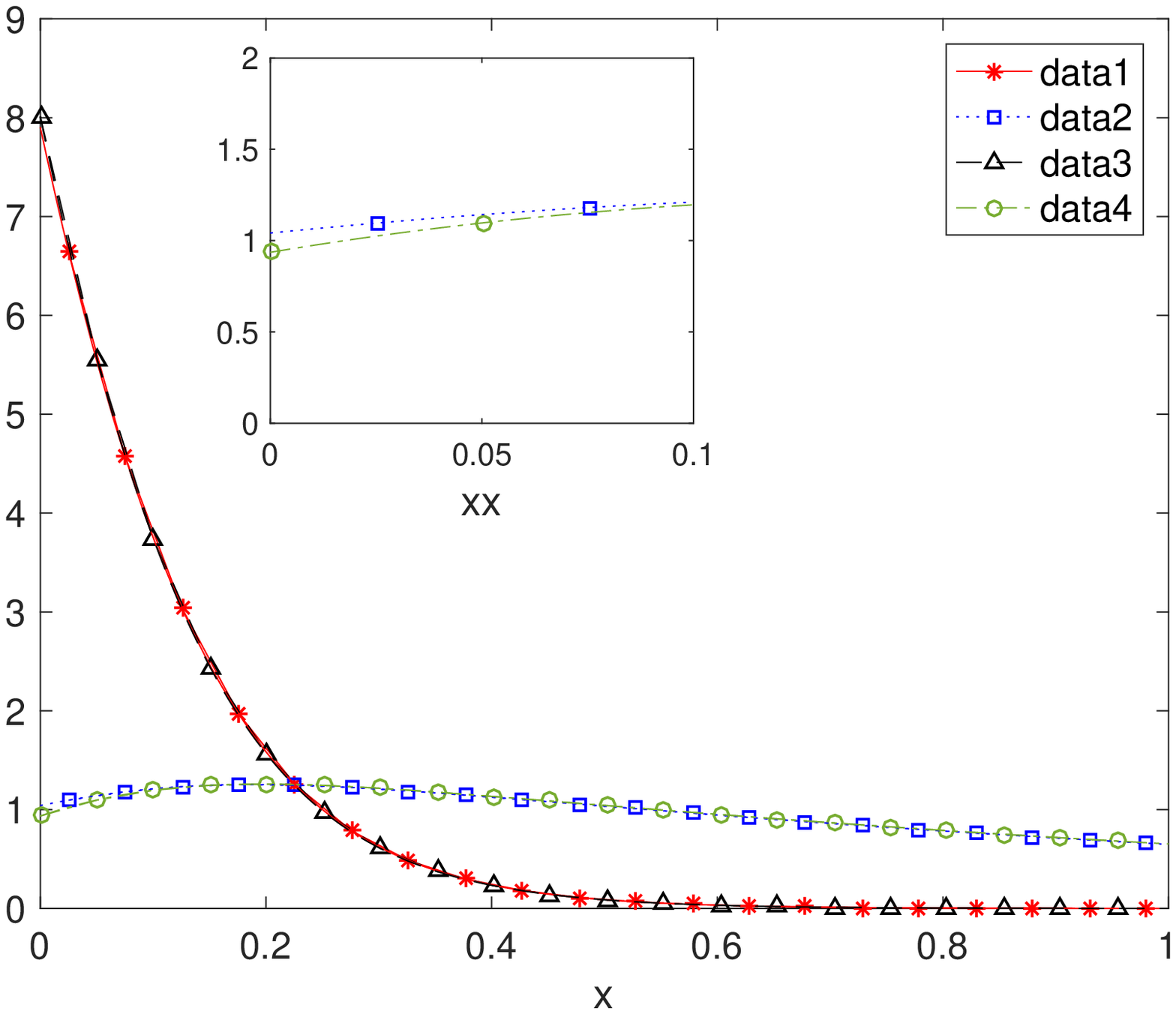}
\caption{Stationary solutions $(r_*, b_*)$ and $(r_\infty, b_\infty)$ from solving the long-time limit of \eqref{pde_general} and \eqref{stationary}, respectively, in a case with $\theta_r = 8 \cdot 10^{-5}$. The parameter values are $d=2$, $D_r = 0.2$, $D_b = 1$, $\epsilon_r = \epsilon_b = 0.01$, $N_b = N_r = 200$ and $\tilde V_r = 2x$, $\tilde V_b = x$.}
\label{fig:stat_error}
\end{center}
\end{figure}

The absolute error and the relative error between the solutions, $\| r_\infty-r_* \|$ and $\| b_\infty - b_*\|$ and $\| r_\infty - r_* \|/ \| r_\infty\|$ and $\| b_\infty - b_*\|/\|b_\infty\|$, respectively, as a function of $\theta_r$ is shown in \figref{fig:errors}. 
\def \scc {0.8}
\def \scl {1.0}
\begin{figure}[ht]
\unitlength=1cm
\begin{center}
\psfragscanon
\psfrag{thetar}[][][\scl]{$\theta_r$} \psfrag{data1}[][][\scc]{$b$}  \psfrag{data2}[][][\scc]{$r$}  \psfrag{Norm}[][][\scl]{Abs. error} \psfrag{norm}[][][\scl]{Rel. error} \psfrag{a}[][][\scl]{(a)} \psfrag{b}[][][\scl]{(b)}
\includegraphics[width = .48\linewidth]{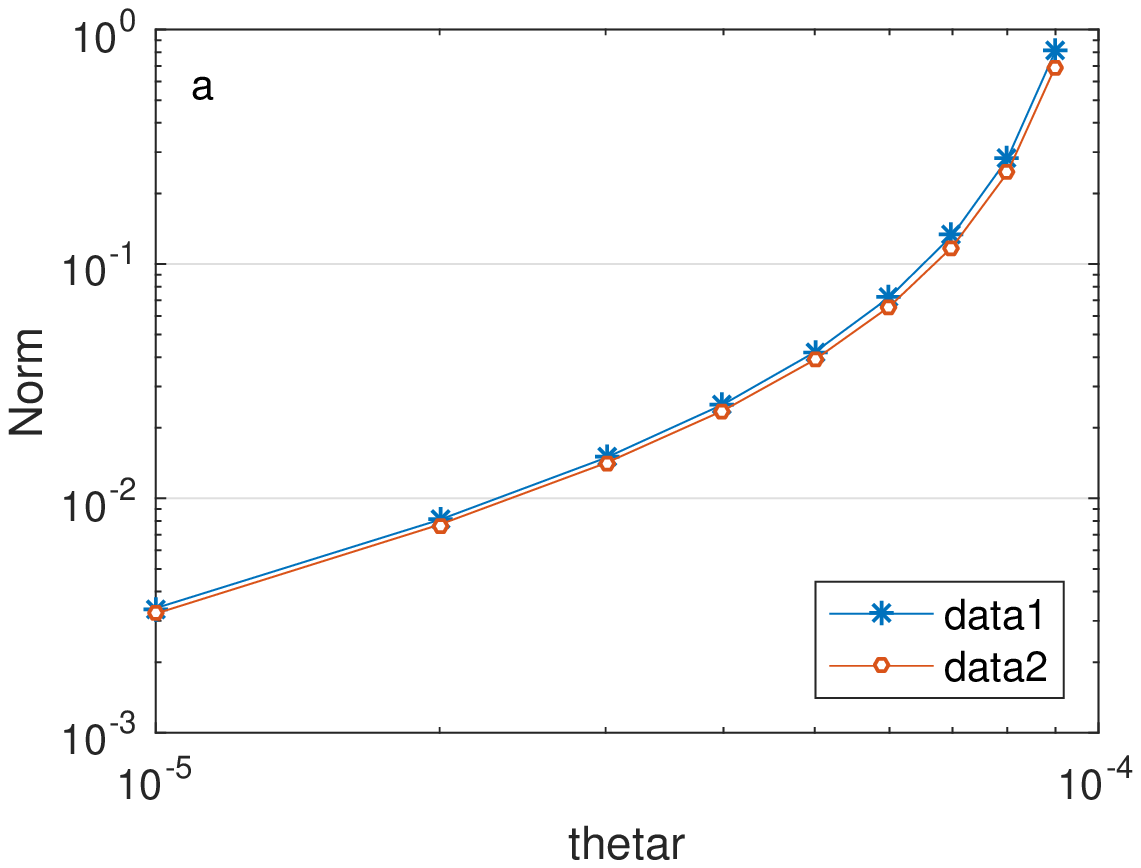} \quad\includegraphics[width = .48\linewidth]{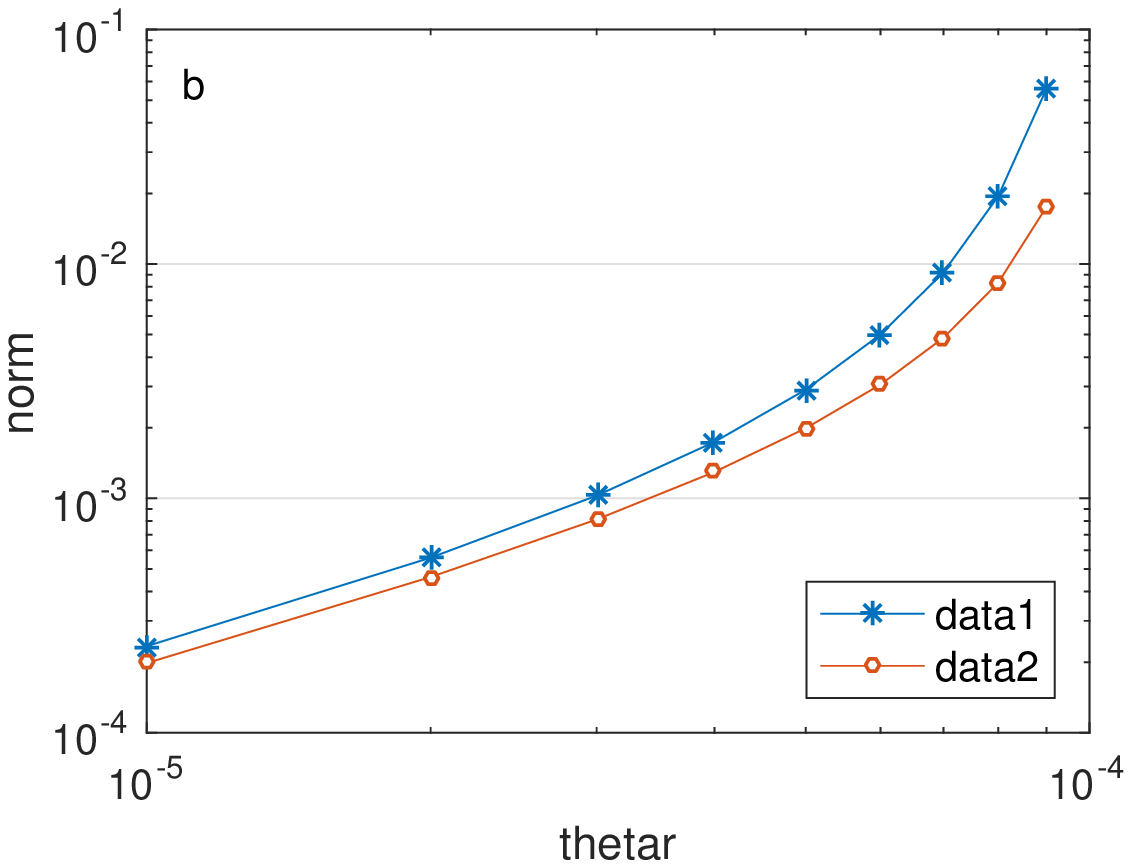}
\caption{Error between the stationary solution $(r_*, b_*)$ of \eqref{pde_general} and $(r_\infty, b_\infty)$ of \eqref{stationary} as a function of $\theta_r$. (a) Absolute error. (b) Relative error. The red particles diffusion $D_r$ is varied according to \eqref{thetar_thetab}, while the other parameter values are fixed to: $d=2$, $D_b = 1$, $\epsilon_r = \epsilon_b = 0.01$, $N_b = N_r = 200$ and $\tilde V_r = 2x$, $\tilde V_b = x$.}
\label{fig:errors}
\end{center}
\end{figure}

To conclude this section, we compute the stationary solutions of the (exact) full system and that approximated by the asymptotic gradient flow system as we vary $\epsilon$, where $\epsilon =\epsilon_b = \epsilon_r$, while keeping all the other parameters fixed. We plot the results in \figref{fig:errors_ep}. As expected from Corollary \ref{corollary_closeness}, the errors scale with $\epsilon^{2d} = \epsilon^4$.

\def \scc {0.8}
\def \scl {1.0}
\begin{figure}[ht]
\unitlength=1cm
\begin{center}
\psfragscanon
\psfrag{ep}[][][\scl]{$\epsilon$} \psfrag{data1}[][][\scc]{$b$}  \psfrag{data2}[][][\scc]{$r$} 
\psfrag{data3}[][][\scc]{$\epsilon^4$} \psfrag{Norm}[][][\scl]{Abs. error} \psfrag{norm}[][][\scl]{Rel. error} \psfrag{a}[][][\scl]{(a)} \psfrag{b}[][][\scl]{(b)}
\includegraphics[width = .48\linewidth]{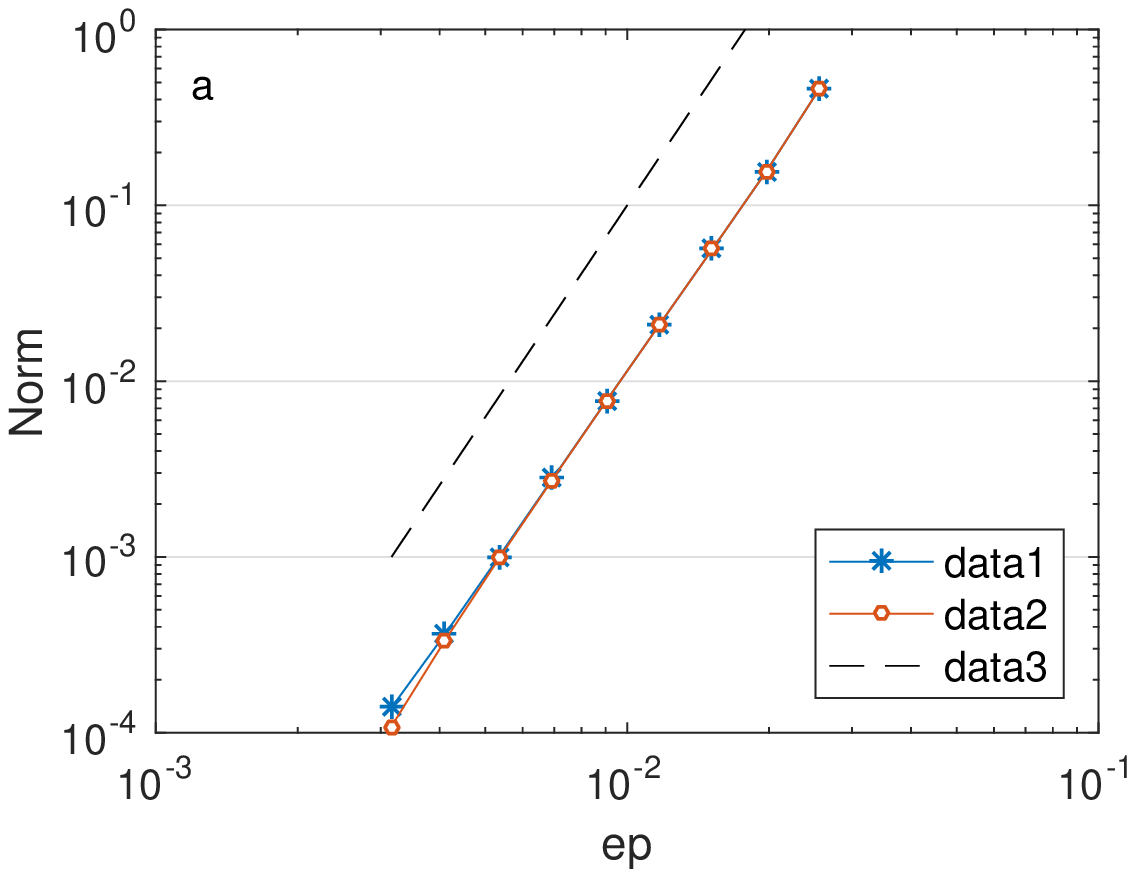} \quad\includegraphics[width = .48\linewidth]{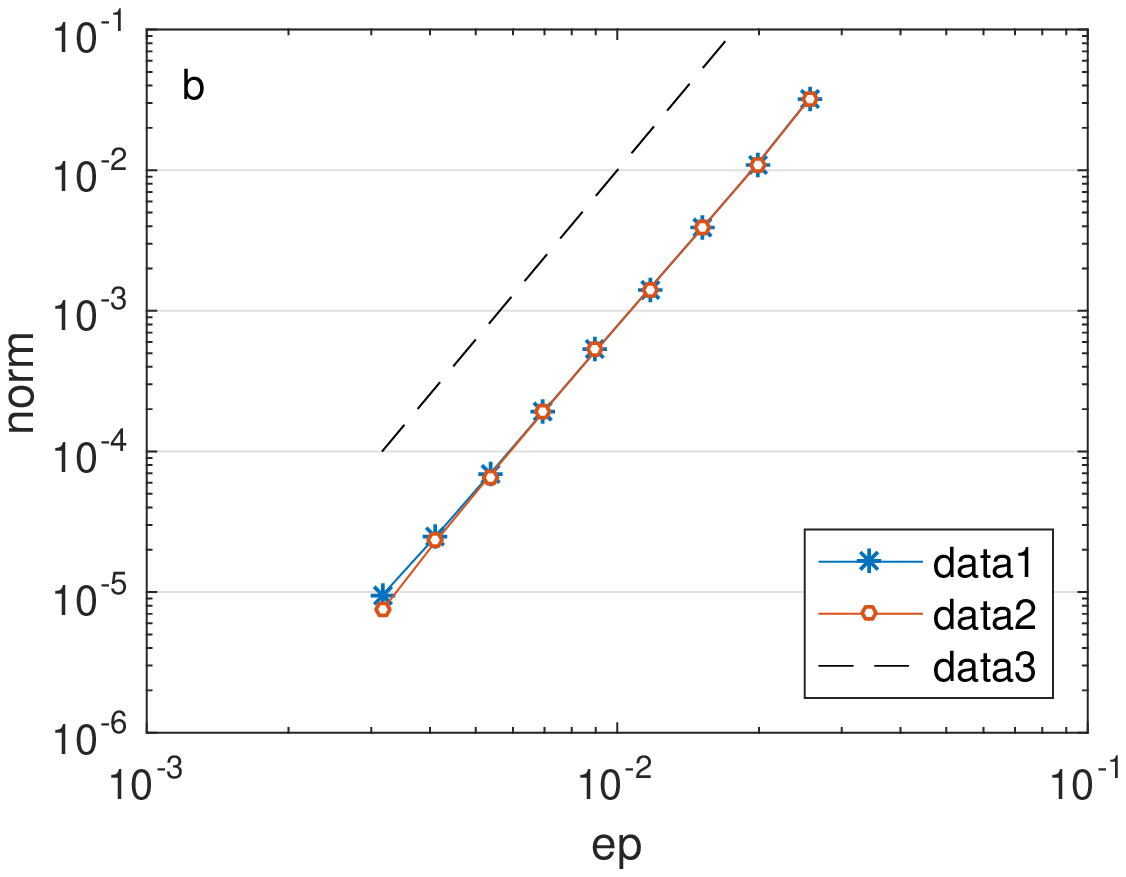}
\caption{Error between the stationary solution $(r_*, b_*)$ of \eqref{pde_general} and $(r_\infty, b_\infty)$ of \eqref{stationary} as a function of $\epsilon$, where $\epsilon = \epsilon_r = \epsilon_b $. (a) Absolute error. (b) Relative error. The parameter values are fixed to: $d=2$, $D_r = 2$, $D_b = 1$, $N_b = N_r = 200$ and $\tilde V_r = 2x$, $\tilde V_b = x$.}
\label{fig:errors_ep}
\end{center}
\end{figure}

\section{Global existence for the full gradient flow system} \label{sec:existence}

In this section we present a global in time existence result for the system with particles of same size and diffusivity \eqref{case1_rho}. 

\begin{theorem}[Global existence in the case of small volume fraction]
Let $T>0$, let $( r_0, b_0):\Omega \to \mathcal{S}^\circ$, where $\mathcal{S}$ is defined by \eqref{equ:set}, be a measurable function such that $E( r_0, b_0)<\infty$. Then there exists a weak solution $( r, b):\Omega\times (0,T)\to \mathcal{S}$ to system 
\begin{align}\label{theorem1_1}
\begin{aligned}
\partial_t
\begin{pmatrix}
 r\\ b
\end{pmatrix} &=\nabla \cdot\begin{pmatrix}
J_r\\J_b
\end{pmatrix} \quad \text{ with }\\
(1-\overline{\gamma}\rho)J_r&=
(1-\overline{\gamma}\rho)\left((1-\bar \gamma \rho)\nabla  r+(\bar \alpha+\bar \gamma) r \nabla \rho + r\nabla V_r+\bar \gamma \nabla(V_b-V_r)rb\right)\\
(1-\overline{\gamma}\rho)J_b&=(1-\overline{\gamma}\rho)\left((1-\bar \gamma \rho)\nabla  b+(\bar \alpha+\bar \gamma) b \nabla \rho + b\nabla V_b+\bar \gamma \nabla(V_b-V_r)rb\right),
\end{aligned}
\end{align}
satisfying 
\begin{align*}
&\partial_t  r,\, \partial_t  b \in L^2(0,T;H^1(\Omega)'),\\
& \rho\, \in L^2(0,T;H^1(\Omega )),\\
&(1-\bar \gamma  \rho)^2\nabla\sqrt{ r},\,(1-\bar \gamma  \rho)^2\nabla\sqrt{ b} \,\in L^2(0,T;L^2(\Omega)).
\end{align*}
Moreover, the solution satisfies the following entropy dissipation inequality:
\begin{align}\label{theorem1_2}
\begin{aligned}
\frac{\mathrm{d}E}{\mathrm{d}t} +\mathcal{D}_1\leq C,
\end{aligned}
\end{align}
where
\begin{equation*}
\mathcal{D}_1=\int_{\Omega} 2(1-\bar \gamma \rho)^4|\nabla\sqrt{ r}|^2 +2(1-\bar \gamma  \rho)^4|\nabla\sqrt{ b}|^2+\frac{\bar \gamma}{2}|\nabla \rho|^2\,d{\bf x}
\end{equation*}
and $C\geq0$ is a constant.
\end{theorem}

We recall that system \eqref{case1_rho} can be written as a gradient flow:

\begin{align} \label{case1_1}
\begin{aligned}
\partial_t
\begin{pmatrix}
 r \\  b
\end{pmatrix}&=\nabla\cdot\left( M( r, b)\nabla \begin{pmatrix}
 u\\  v
\end{pmatrix}\right),
\end{aligned}
\end{align} 
where 
\[M=\begin{pmatrix}
 r(1-\bar \gamma  b)& \bar \gamma  r b\\
\bar \gamma  r b &  b (1-\bar \gamma  r)\\
\end{pmatrix}.\]
Note that if $ r, b$ and $\rho \in \mathcal{S}^\circ$, then the matrix $M$ is positive definite.

We perform a time discretisation of system \eqref{case1_1} using the implicit Euler scheme. The resulting recursive sequence of elliptic problems is then regularized.
Let $N\in\mathbb{N}$ and let $\tau=T/N$ be the time step size. We split the time interval into the subintervals
\[(0,T]=\bigcup_{k=1}^N ((k-1)\tau,k\tau],\qquad \tau=\frac{T}{N}.\] Then for given functions $( r_{k-1},  b_{k-1}) \in \mathcal{S}$, which approximate $( r, b)$ at time $\tau(k-1)$, we want to find $( r_k, b_k) \in \mathcal{S}$ solving the regularized time discrete problem
\begin{align}
\label{case1_1_reg}
\begin{aligned}
\frac{1}{\tau}\begin{pmatrix}
 r_k- r_{k-1} \\  b_k- b_{k-1}
\end{pmatrix}&=\nabla\cdot\left( M( r_k, b_k)\begin{pmatrix}
\nabla \tilde{u}_k\\ \nabla \tilde{v}_k\end{pmatrix}\right)+\tau\begin{pmatrix}
 \Delta \tilde{u}_k-\tilde{u}_k\\ \Delta \tilde{v}_k-\tilde{v}_k
\end{pmatrix},
\end{aligned}
\end{align} 
where we use the modified entropy
\begin{align}
\tilde{E} =E+E_\tau=\int_{\Omega}  &  r(\log   r -1)+    b (\log   b-1)   +  r  V_r +  b V_b + \frac{\bar \alpha}{2} \left(    r^2 + 2   r  b  +    b^2 \right) \\
& + \tau (1-\bar \gamma  \rho) (\log (1-\bar \gamma  \rho)-1)\, d  {\bf x},\nonumber
\end{align}
with associated entropy variables
\begin{align}
\begin{aligned}
\tilde{u}=u+u_\tau &=  \log  r + \bar \alpha  \rho + V_r- \tau \bar \gamma \log (1-\bar \gamma  \rho) ,\\
\tilde{v}=v+v_\tau &=  \log  b + \bar \alpha  \rho + V_b- \tau \bar \gamma \log (1-\bar \gamma  \rho).
\end{aligned}
\end{align}

The additional term in the entropy provides upper bounds on the solutions and the higher order regularization terms guarantee coercivity of the elliptic system in $H^1(\Omega)$, which is needed to show existence of weak solutions to a linearized version of the problem \eqref{case1_1_reg} using Lax-Milgram. 
The existence result of the corresponding nonlinear problem is concluded by applying Schauder fixed point theorem.

Finally uniform a priori estimates in $\tau$ and the use of a generalized version of the Aubin-Lions lemma allow to pass to the limit $\tau \to 0$ leading to the existence of \eqref{theorem1_1}. 
Note that the compactness results are sufficient for $1-\bar \gamma\rho >0$ to pass to the correct limit in the flux terms $J_r$ and $J_b$, i.e leading to the global existence of weak solutions to system \eqref{case1_rho}.

\begin{lemma}\label{lemma1}
The entropy density 
\begin{align*}
\tilde{h}:\mathcal{S}^\circ\to \mathbb{R}, \begin{pmatrix}
r\\b
\end{pmatrix}& \mapsto r(\log   r-1) +    b (\log   b-1)   +  r  V_r +  b V_b \\
&+ \frac{\bar \alpha}{2} \left(    r^2 + 2   r  b  +    b^2 \right) + \tau (1-\bar \gamma  \rho) (\log (1-\bar \gamma  \rho)-1)
\end{align*}
is strictly convex and belongs to $C^2(\mathcal{S}^\circ).$ Its gradient $\tilde{h}':\mathcal{S}^\circ\to \mathbb{R}^2$ is invertible and the inverse of the Hessian $\tilde{h}'':\mathcal{S}^\circ\to \mathbb{R}^{2\times 2}$ is uniformly bounded.
\end{lemma}

\begin{proof}
Note that
\[\tilde{h}'=\begin{pmatrix}
\log  r-\tau\bar \gamma \log (1-\bar \gamma  \rho) +\bar \alpha \rho +V_r\\
\log  b-\tau\bar \gamma \log (1-\bar \gamma  \rho) +\bar \alpha \rho +V_b
\end{pmatrix}\]
and
\[\tilde{h}''=\begin{pmatrix}
\frac{1}{ r}+\tau\frac{\bar \gamma^2}{1-\bar \gamma  \rho}+\bar \alpha & \tau\frac{\bar \gamma^2}{1-\bar \gamma  \rho}+\bar \alpha \\
\tau \frac{\bar \gamma^2}{1-\bar \gamma  \rho}+\bar \alpha  &\frac{1}{ b} +\tau \frac{\bar \gamma^2}{1-\bar \gamma  \rho}+\bar \alpha 
\end{pmatrix}.\]
The matrix $\tilde{h}''$ is positive definite on the set $\mathcal{S}^\circ$, so $\tilde{h}$ is strictly convex. We can easily deduce that the inverse of $\tilde{h}''$ exists and is bounded on $\mathcal{S}^\circ$.

Next we verify the invertibility of $\tilde{h}'$.
Note that the function $g=(g_1,g_2):\mathcal{S}^\circ\to \mathbb{R}^2,( r, b)\mapsto (\log  r-\tau \bar \gamma \log (1-\bar \gamma  \rho),
\log  b-\tau \bar \gamma \log (1-\bar \gamma  \rho))$ is invertible. Let $(x,y)\in \mathbb{R}^2$ and define $u(z)=(e^x+e^y)(1-\bar \gamma z)$ for $0<z<\frac{1}{\bar \gamma}$. Then $u$ is nonincreasing and as $u(0)>0$ and $u\left(\frac{1}{\bar \gamma}\right)=0$, there exists a unique fixed point $0<z_0<\frac{1}{\bar \gamma}$ such that $u(z_0)=z_0$. Then we define $ r=e^x(1-\bar \gamma z_0)>0$ and $ b =e^y(1-\bar \gamma z_0)>0$. It holds that $ r+ b=(e^x+e^y)(1-\bar \gamma z_0)=z_0<\frac{1}{\bar \gamma}$. So, $( r, b)\in \mathcal{S}^\circ$.
Then, we define the function $f=\tilde{h}'\circ g^{-1}:\mathbb{R}^2\to \mathbb{R}^2$. Since $\tilde{h}''$ and $g'$ are nonsingular matrices for $( r , b)\in \mathcal{S}^\circ$, the Jacobian of $f$ is also nonsingular for $( r , b)\in \mathcal{S}^\circ$.
Furthermore, we have that 
\[f(y)=y+\chi (g^{-1}(y)),\quad y\in \mathbb{R}^2,\]
where $\chi=\begin{pmatrix}
\bar \alpha  \rho +V_r\\
\bar \alpha  \rho +V_b
\end{pmatrix}\in C^0(\mathcal{S})\subseteq L^{\infty}(\mathbb{\mathcal{S}^\circ})$.
So, $|f(y)|\to \infty$ as $|y|\to \infty$, which together with the invertibility of the matrix $Df$ allow us to apply Hadamard's global inverse theorem showing that $f$ is invertible. So, also $\tilde{h}'$ is invertible.
\end{proof}

\subsection{Time discretisation and regularization of system \eqref{case1_1}}

The weak formulation of system \eqref{case1_1_reg} is given by: 

\begin{align}
\label{case1_1_reg_weak}
\begin{aligned}
\frac{1}{\tau}\int_\Omega \begin{pmatrix}
 r_k- r_{k-1} \\  b_k- b_{k-1}
\end{pmatrix} \cdot \begin{pmatrix}
\Phi_1 \\ \Phi_2
\end{pmatrix}\, d{\bf x}&+\int_\Omega\begin{pmatrix}
\nabla \Phi_1 \\ \nabla \Phi_2 
\end{pmatrix}^T M( r_k, b_k)\begin{pmatrix}
\nabla \tilde{u}_k\\ \nabla \tilde{v}_k\end{pmatrix}\,d{\bf x}\\
&+\tau R\left(\begin{pmatrix}
\Phi_1\\ \Phi_2
\end{pmatrix},\begin{pmatrix}
\tilde{u}_k\\\tilde{v}_k
\end{pmatrix}\right)=0
\end{aligned}
\end{align} 

for $(\Phi_1,\Phi_2)\in H^1(\Omega)\times H^1(\Omega)$, where $( r_k, b_k)=h'^{-1}(\tilde{u}_k,\tilde{v}_k)$ and
\begin{align*}
\begin{aligned}
R\left(\begin{pmatrix}
\Phi_1\\ \Phi_2
\end{pmatrix},\begin{pmatrix}
\tilde{u}_k\\\tilde{v}_k
\end{pmatrix}\right)&=\int_{\Omega}
\Phi_1\tilde{u}_k+\Phi_2\tilde{v}_k+\nabla \Phi_1\cdot \nabla \tilde{u}_k+\nabla \Phi_2\cdot \nabla \tilde{v}_k
\,dx\,dy.
\end{aligned}
\end{align*}

We define $F:\mathcal{S}\subseteq L^2(\Omega,\mathbb{R}^2)\to\mathcal{S}\subseteq L^2(\Omega,\mathbb{R}^2), (\tilde{r} ,\tilde{b}) \mapsto ( r, b)=h'^{-1}(\tilde{u},\tilde{v})$, where
$(\tilde{u},\tilde{v})$ is the unique solution in $H^1(\Omega,\mathbb{R}^2)$ to the linear problem
\begin{equation} \label{equ_lax} 
a((\tilde{u},\tilde{v}),(\Phi_1,\Phi_2))=F(\Phi_1,\Phi_2) \quad \text{for all }(\Phi_1,\Phi_2)\in H^1(\Omega,\mathbb{R}^2)
\end{equation}
with 
\begin{align*}
a((\tilde{u},\tilde{v}),(\Phi_1,\Phi_2))&=\int_{\Omega}\begin{pmatrix}
\nabla \Phi_1\\ \nabla \Phi_2
\end{pmatrix}^T M(\tilde{r},\tilde{b})\begin{pmatrix}
\nabla u \\ \nabla v 
\end{pmatrix}\,d{\bf x}+\tau R\left(\begin{pmatrix}
\Phi_1\\ \Phi_2
\end{pmatrix},\begin{pmatrix}
\tilde{u}\\\tilde{v}
\end{pmatrix}\right)\\
F(\Phi_1,\Phi_2)&=-\frac{1}{\tau}\int_{\Omega}\begin{pmatrix}
\tilde{r}- r_{k-1}\\\tilde{b}- b_{k-1}
\end{pmatrix}\cdot\begin{pmatrix}
\Phi_1\\\Phi_2
\end{pmatrix}\,d{\bf x}
\end{align*}
The bilinear form $a:H^1(\Omega;\mathbb{R}^2)\times H^1(\Omega;\mathbb{R}^2)\to \mathbb{R}$ and the functional $F:H^1(\Omega,\mathbb{R}^2)\to \mathbb{R}$ are bounded. Moreover, $a$ is coercive since the positive semi-definiteness of $M(r,b)$ implies that
\begin{align*}
a((\tilde{u},\tilde{v}),(\tilde{u},\tilde{v}))&=\int_{\Omega}\begin{pmatrix}
\nabla \tilde{u}\\ \nabla \tilde{v}
\end{pmatrix}^T M(\tilde{r},\tilde{b})\begin{pmatrix}
\nabla \tilde{u} \\ \nabla \tilde{v} 
\end{pmatrix}\,d{\bf x}+\tau R\left(\begin{pmatrix}
\tilde{u}\\ \tilde{v}
\end{pmatrix},\begin{pmatrix}
\tilde{u}\\ \tilde{v}
\end{pmatrix}\right)\\
&\geq \tau \left(\|\tilde{u}\|_{H^1(\Omega)}^2+\|\tilde{v}\|_{H^1(\Omega)}^2\right).
\end{align*}
Then the Lax-Milgram lemma guarantees the existence of a unique solution $(\tilde{u},\tilde{v})\in H^1(\Omega;\mathbb{R}^2)$ to \eqref{equ_lax}.

To apply Schauer's fixed point theorem, we need to show that the map $S$:
\begin{compactenum}[(i)]
\item  maps a convex, closed set onto itself,\label{schauder1}
\item  is compact,\label{schauder2}
\item  is continuous.\label{schauder3}
\end{compactenum}
Since $\mathcal{S}$ is convex and closed, property \eqref{schauder1} is satisfied; \eqref{schauder2} follows from the compact embedding $H^1(\Omega,\mathbb{R}^2)\hookrightarrow L^2(\Omega,\mathbb{R}^2)$. 
Continuity \eqref{schauder3}: let $(\tilde{r}_k,\tilde{b}_k)$ be a sequence in $\mathcal{S}$ converging strongly to $(\tilde{r},\tilde{b})$ in $L^2(\Omega,\mathbb{R}^2)$ and let $(\tilde{u}_k,\tilde{v}_k)$ be the corresponding unique solution to \eqref{equ_lax} in $H^1(\Omega;\mathbb{R}^2)$. As the matrix $M$ only contains sums and products of $ r$ and $ b$, we have that $M(\tilde{r}_k,\tilde{b}_k)\to M(\tilde{r},\tilde{b})$ strongly in $L^2(\Omega,\mathbb{R}^2)$. The positive semidefiniteness of the matrix $M$ for $( r, b)\in \mathcal{S}$ provides a uniform bound for $(\tilde{u}_k,\tilde{v}_k)$ in $H^1(\Omega;\mathbb{R}^2)$. Hence, there exists a subsequence with $(\tilde{u}_k,\tilde{v}_k)\rightharpoonup (\tilde{u},\tilde{v})$ weakly in $H^1(\Omega;\mathbb{R}^2)$. The $L^{\infty}$ bounds of $M(\tilde{r}_k,\tilde{b}_k)$ and the application of a density argument allow us to pass from test functions $(\Phi_1,\Phi_2)\in W^{1,\infty}(\Omega,\mathbb{R}^2)$ to test functions $(\Phi_1,\Phi_2)\in H^1(\Omega,\mathbb{R}^2)$. So, the limit $(\tilde{u},\tilde{v})$ as the solution of problem \eqref{equ_lax} with coefficients $(\tilde{r},\tilde{b})$ is well defined. 
Due to the compact embedding $H^1(\Omega,\mathbb{R}^2)\hookrightarrow L^2(\Omega,\mathbb{R}^2)$, we have a strongly converging subsequence of $(\tilde{u}_k,\tilde{v}_k)$ in $L^2(\Omega,\mathbb{R}^2)$. Since the limit is unique, the whole sequence converges. From Lemma \ref{lemma1} we know that $( r, b)=h'^{-1}(\tilde{u},\tilde{v})$ is Lipschitz continuous, which yields continuity of $F$.

Hence, we can apply Schauder's fixed point theorem, which assures the existence of a solution $( r, b)\in \mathcal{S}$ to  \eqref{equ_lax} with $(\tilde{r},\tilde{b})$ replaced by $(r, b)$.

\subsection{Entropy dissipation} \label{sec:entropy_diss}

\begin{lemma}\label{lemma2}
Let $ r,  b :\Omega \rightarrow \mathcal{S}$ be a sufficiently smooth solution to system
\begin{align}
\label{case1_1_ent}
\begin{aligned}
\partial_t
\begin{pmatrix}
 r \\  b
\end{pmatrix}&=\nabla\cdot\left( M( r, b) \nabla \begin{pmatrix}
 \tilde{u}\\  \tilde{v}
\end{pmatrix}\right).
\end{aligned}
\end{align} 
Then, the entropy $\tilde{E}$ is decreasing and there exists a constant $C\geq 0$ such that
\begin{align}
\label{entropyinequality}
\begin{aligned}
\frac{\mathrm{d} \tilde{E}}{\mathrm{d}t}+\mathcal{D}_0&\leq C,
\end{aligned}
\end{align}
where 
\begin{equation*}
\mathcal{D}_0=\int_{\Omega} 2(1-\bar \gamma  \rho)|\nabla\sqrt{ r}|^2 +2(1-\bar \gamma  \rho)|\nabla\sqrt{ b}|^2+\frac{\bar \gamma}{2}|\nabla  \rho|^2+\frac{\tau^2}{2} \frac{\bar \gamma^5  \rho^2}{(1-\bar \gamma  \rho)^2}|\nabla  \rho |^2\,d{\bf x}.
\end{equation*}
\end{lemma}

\begin{proof}
System \eqref{case1_1_ent} enables us to deduce the entropy dissipation relation:
\begin{align}\label{equ_1}
\begin{aligned}
\frac{\mathrm{d}\tilde{E}}{\mathrm{d}t} &=\int_{\Omega} (\tilde{u}\, \partial_t  r +\tilde{v}\, \partial_t  b ) \,d{\bf x} =-\int_{\Omega} \begin{pmatrix}
\nabla \tilde{u} \\ \nabla \tilde{v}
\end{pmatrix}^T M \begin{pmatrix}
\nabla \tilde{u} \\ \nabla \tilde{v}
\end{pmatrix}\,d{\bf x}\\
&=-\int_{\Omega}   r (1 - \bar \gamma   b) |\nabla \tilde{u}|^2+  b (1 - \bar \gamma   r) |\nabla \tilde{v}|^2 +2\bar \gamma  r  b \nabla \tilde{u}\nabla \tilde{v}\,d{\bf x}\\
&=-\int_{\Omega}   r (1 - \bar \gamma   \rho) |\nabla \tilde{u}|^2+  b (1 - \bar \gamma   \rho) |\nabla \tilde{v}|^2 +\bar \gamma | r \nabla \tilde{u}+ b \nabla\tilde{v}|^2\,d{\bf x}\leq 0.
\end{aligned}
\end{align}
Inequality \eqref{entropyinequality} follows from the definitions of $\tilde{u}$ and $\tilde{v}$ as well as Young's inequality to estimate the mixed terms. Furthermore we use that
\begin{align*}
\quad  r(1-\bar \gamma  \rho)& \left|\frac{\nabla  r}{ r}+\tau\frac{\bar \gamma^2}{1-\bar \gamma  \rho}\nabla  \rho +\bar \alpha \nabla  \rho \right|^2+ b (1-\bar \gamma  \rho) \left| \frac{\nabla  b}{ b}+\tau\frac{\bar \gamma^2}{1-\bar \gamma  \rho}\nabla  \rho +\bar \alpha \nabla \rho\right|^2\\
&=4(1-\bar \gamma  \rho)|\nabla \sqrt{ r}|^2+4(1-\bar \gamma  \rho)|\nabla \sqrt{ b}|^2+\bar \alpha^2 \rho (1-\bar \gamma  \rho)|\nabla  \rho|^2+2\bar \alpha(1-\bar \gamma\rho)|\nabla\rho|^2\\
&\quad +\tau^2 \frac{\bar \gamma^4  \rho}{1-\bar \gamma  \rho}|\nabla  \rho |^2+2\tau \bar \gamma^2|\nabla  \rho|^2+2\tau  \rho \bar \gamma^2 \bar \alpha |\nabla  \rho |^2 
\end{align*}
and
\begin{align*}
&\quad \bar \gamma | r \nabla \tilde{u}+ b \nabla\tilde{v}|^2 = \bar \gamma \left|\nabla  \rho \left(1+\frac{\tau \bar \gamma^2  \rho }{1-\bar \gamma\rho}+\bar \alpha \rho \right)+ r\nabla V_r+ b \nabla V_b \right|^2.
\end{align*}

This gives us
\begin{align*}
\frac{\mathrm{d}E}{\mathrm{d}t} &\leq -\int_\Omega  2(1-\bar \gamma  \rho)|\nabla \sqrt{ r}|^2+2(1-\bar \gamma  \rho)|\nabla \sqrt{ b}|^2+ \frac{\bar \gamma}{2}|\nabla\rho|^2+\frac{\tau^2}{2} \frac{\bar \gamma^5  \rho^2}{(1-\bar \gamma  \rho)^2}|\nabla  \rho |^2\, d{\bf x} \\
&\quad  +\int_\Omega (1-\bar \gamma  \rho)( r|\nabla V_r|^2+ b|\nabla V_b|^2)+\bar \gamma | r \nabla V_r + b \nabla V_b|^2\, d{\bf x}.
\end{align*}
Since $ r,  b$ and $\rho \in \mathcal{S}$ and $\nabla V_r, \nabla V_b\in L^1(\Omega)$, we deduce \eqref{entropyinequality}.
\end{proof}

\subsection{The limit $\tau \to 0$.} \label{sec:limit_tau}

As the entropy density $\tilde{h}$ is convex, we have $\tilde{h}(\varphi_1)-\tilde{h}(\varphi_2)\leq \tilde{h}'(\varphi_1)\cdot(\varphi_1-\varphi_2)$ for all $\varphi_1,\varphi_2\in\mathcal{S}$. Choosing $\varphi_1=( r_k, b_k)$ and $\varphi_2=( r_{k-1}, b_{k-1})$ and using $\tilde{h}'( r_k, b_k)=( \tilde{u}_k, \tilde{v}_k)$, we obtain
\begin{align}\label{inequ1}
\frac{1}{\tau}\int_{\Omega}&\begin{pmatrix}
 r_k- r_{k-1}\\ b_k- b_{k-1}
\end{pmatrix}\cdot\begin{pmatrix}
 \tilde{u}_k\\ \tilde{v}_k
\end{pmatrix}\,d{\bf x}\geq\frac{1}{\tau}\int_{\Omega}\begin{pmatrix}
\tilde{h}( r_k, b_k)-\tilde{h}( r_{k-1},  b_{k-1})
\end{pmatrix}\,d{\bf x}.
\end{align}
Applying \eqref{inequ1} in equation \eqref{case1_1_reg_weak} with the test function $(\Phi_1,\Phi_2)=(\tilde{u}_k,\tilde{v}_k)$ leads to
\begin{align}\label{inequ2}
\begin{aligned}
\int_{\Omega}\tilde{h}( r_k, b_k)\,d{\bf x}
+\tau\int_{\Omega}\begin{pmatrix}
\nabla \tilde{u}_k\\ \nabla \tilde{v}_k
\end{pmatrix}^T M( r_k, b_k)\begin{pmatrix}
\nabla \tilde{u}_k \\ \nabla \tilde{v}_k 
\end{pmatrix}\,d{\bf x} \\
+\tau^2 R \left(\begin{pmatrix}
\tilde{u}_k\\\tilde{v}_k
\end{pmatrix},\begin{pmatrix}
\tilde{u}_k\\\tilde{v}_k
\end{pmatrix}\right)&\leq\int_{\Omega}\tilde{h}( r_{k-1}, b_{k-1})\,d{\bf x}.
\end{aligned}
\end{align}
Applying the entropy inequality \eqref{entropyinequality} and resolving recursion \eqref{inequ2} yields
\begin{align}\label{discrete_entropyinequality}
\begin{aligned}
\quad\int_{\Omega} \tilde{h}( r_k, b_k)\,d{\bf x}+&\tau\sum_{j=1}^k\int_{\Omega} 2(1-\bar \gamma  \rho_j)|\nabla\sqrt{ r_j}|^2 +2(1-\bar \gamma  \rho_j)|\nabla\sqrt{ b_j}|^2+\frac{\bar \gamma}{2}|\nabla  \rho_j|^2\\
&+\frac{\tau^2}{2} \frac{\bar \gamma^5  \rho_j^2}{(1-\bar \gamma  \rho_j)^2}|\nabla  \rho_j |^2\,d{\bf x}+\tau^2\sum_{j=1}^k R \left(\begin{pmatrix}
\tilde{u}_j\\\tilde{v}_j
\end{pmatrix},\begin{pmatrix}
\tilde{u}_j\\\tilde{v}_j
\end{pmatrix}\right) \\
&\leq \int_{\Omega} \tilde{h}( r_0, b_0)\,dx\,dy+T C.
\end{aligned}
\end{align}

Let $( r_k, b_k)$ be a sequence of solutions to \eqref{case1_1_reg_weak}. We define $ r_\tau({\bf x},t)= r_k({\bf x})$ and $b_\tau({\bf x},t)=b_k({\bf x})$ for ${\bf x}\in\Omega$ and $t\in ((k-1)\tau,k\tau]$. Then $( r_\tau, b_\tau)$ solves the following problem, where $\sigma_\tau$ denotes a shift operator, i.e. $(\sigma_\tau  r_\tau)({\bf x},t)= r_\tau ({\bf x},t-\tau)$ and $(\sigma_\tau  b_\tau)({\bf x},t)= b_\tau ({\bf x},t-\tau)$ for $\tau \leq t\leq T$,

\begin{align}\label{equ1_tau}
&\int_0^T\int_{\Omega}\frac{1}{\tau}\begin{pmatrix}
 r_\tau-\sigma_\tau  r_\tau\\ b_\tau-\sigma_\tau  b_\tau
\end{pmatrix}\cdot\begin{pmatrix}
\Phi_1\\\Phi_2
\end{pmatrix}+
\begin{pmatrix}
(1-\bar \gamma \rho_\tau)\nabla  r_\tau+(\bar \alpha+\bar \gamma) r_\tau \nabla \rho_\tau \\
(1-\bar \gamma \rho_\tau)\nabla  b_\tau+(\bar \alpha+\bar \gamma) b_\tau \nabla \rho_\tau \\
\end{pmatrix}\cdot\begin{pmatrix}
\nabla \Phi_1\\ \nabla \Phi_2
\end{pmatrix}\,d{\bf x}\,dt\nonumber \\
&\qquad \qquad +\int_0^T \int_{\Omega}\begin{pmatrix}
 r_\tau\nabla V_r+\bar \gamma \nabla(V_b-V_r) r_\tau  b_\tau\\
 b_\tau\nabla V_b+\bar \gamma \nabla(V_r-V_b) r_\tau  b_\tau\\
\end{pmatrix}\cdot\begin{pmatrix}
\nabla \Phi_1\\ \nabla \Phi_2
\end{pmatrix}d{\bf x}\,dt\\
&\qquad \qquad +\int_0^T \int_{\Omega}\begin{pmatrix}
\frac{\tau\bar \gamma^2r_\tau}{1-\bar \gamma\rho_\tau}\nabla \rho_\tau \\
\frac{\tau\bar \gamma^2 b_\tau}{1-\bar \gamma\rho_\tau}\nabla \rho_\tau
\end{pmatrix}\cdot \begin{pmatrix}
\nabla \Phi_1\\ \nabla \Phi_2
\end{pmatrix}d{\bf x}+\tau R\left(\begin{pmatrix}
\Phi_1\\\Phi_2
\end{pmatrix},\begin{pmatrix}
\tilde{u}_\tau\\\tilde{v}_\tau
\end{pmatrix}\right)\,dt=0,\nonumber
\end{align}
for $(\Phi_1(t),\Phi_2(t))\in L^2(0,T;H^1(\Omega))$. Note that the terms in the third line are the regularization terms.

Inequality \eqref{discrete_entropyinequality} becomes
\begin{align}\label{entropyinequality2}
\begin{aligned}
\quad\int_{\Omega} \tilde{h}( r_\tau(T), b_\tau(T))\,d{\bf x}&+\int_0^T\int_{\Omega} 2(1-\bar \gamma  \rho_\tau)|\nabla\sqrt{ r_\tau}|^2 +2(1-\bar \gamma  \rho_\tau)|\nabla\sqrt{ b_\tau}|^2+\frac{\bar \gamma}{2}|\nabla  \rho_\tau|^2\\
&+\frac{\tau^2}{2} \frac{\bar \gamma^5  \rho_\tau^2}{(1-\bar \gamma  \rho_\tau)^2}|\nabla  \rho_\tau |^2\,d{\bf x}\,dt+\tau\int_0^T R \left(\begin{pmatrix}
\tilde{u}_\tau\\\tilde{v}_\tau
\end{pmatrix},\begin{pmatrix}
\tilde{u}_\tau\\\tilde{v}_\tau
\end{pmatrix}\right)\,dt\\
&\leq  \int_{\Omega}\tilde{h}(r_0,b_0)\,dx\,dy+T C,
\end{aligned}
\end{align}
which provides us the following a priori estimates. Note that from now on $K$ denotes a generic constant.  
\begin{lemma}{(A priori estimates)}\label{lemma3}
There exists a constant $K\in\mathbb{R}^+$, such that the following bounds hold:
\begin{align}
\|\sqrt{1-\bar \gamma  \rho_\tau}\nabla\sqrt{ r_\tau}\|_{L^2(\Omega_T)}+\|\sqrt{1-\bar \gamma  \rho_\tau}\nabla\sqrt{ b_\tau}\|_{L^2(\Omega_T)}&\leq K, \label{apriori1}\\
\| \rho_\tau\|_{L^2(0,T;H^1(\Omega))}&\leq K,  \label{apriori2}\\
\tau \left(\left\|\frac{ r_\tau}{1-\bar \gamma  \rho_\tau}\nabla  \rho_\tau\right\|_{L^2(\Omega_T)}+ \left\|\frac{ b_\tau}{1-\bar \gamma  \rho_\tau}\nabla  \rho_\tau\right\|_{L^2(\Omega_T)}\right)&\leq K,  \label{apriori3}\\
\sqrt{\tau}(\|\tilde{u}_\tau\|_{L^2(0,T;H^1(\Omega))}+\|\tilde{v}_\tau\|_{L^2(0,T;H^1(\Omega))})&\leq K,\label{apriori4}
\end{align}
where $\Omega_T=\Omega \times (0,T)$.
\end{lemma}
\begin{lemma} \label{lemma4}
The discrete time derivatives of $ r_\tau$ and $ b_\tau$ are uniformly bounded, i.e.
\begin{align}\label{inequ3}
\frac{1}{\tau}\| r_\tau-\sigma_\tau  r_\tau\|_{L^2(0,T;H^1(\Omega)')}+\frac{1}{\tau}\| b_\tau-\sigma_\tau  b_\tau\|_{L^2(0,T;H^1(\Omega)')}&\leq K. 
\end{align}
\end{lemma}
\begin{proof}
Let $\Phi \in L^2(0,T;H^1(\Omega))$. Using the a priori estimates from Lemma \ref{lemma3} gives  
\begin{align*}
\frac{1}{\tau}\int_0^T \langle  r_\tau-\sigma_\tau  r_\tau,\Phi\rangle\,dt &= -\int_0^T\int_{\Omega}((1-\bar \gamma \rho_\tau)\nabla  r_\tau+(\bar \alpha+\bar \gamma) r_\tau \nabla \rho_\tau) \nabla \Phi\,d{\bf x}\,dt\\
&-\int_0^T \int_{\Omega}( r_\tau\nabla V_r+\bar \gamma \nabla(V_b-V_r) r_\tau  b_\tau)\nabla \Phi\,d{\bf x}\,dt\\
&-\tau \bar\gamma^2\int_0^T \int_{\Omega} \frac{r_\tau}{1-\bar \gamma\rho_\tau}\nabla \rho_\tau \nabla \Phi\,d{\bf x}\,dt\\
&-\tau \int_0^T\int_{\Omega} \tilde{u}_\tau\Phi+ \nabla \tilde{u}_\tau\cdot\nabla \Phi\,d{\bf x}\,dt\\
\leq \,&\|(1-\bar \gamma \rho_\tau)\nabla  r_\tau\|_{L^2(\Omega_T)}\|\nabla \Phi\|_{L^2(\Omega_T)}\\
&+(\bar \alpha +\bar \gamma )\| r_\tau \|_{L^{\infty}(\Omega_T)}\|\nabla  \rho_\tau \|_{L^2(\Omega_T)}\|\nabla \Phi\|_{L^2(\Omega_T)}\\
&+\| r_\tau \nabla V_r+\bar \gamma \nabla(V_b-V_r) r_\tau b_\tau\|_{L^{\infty}(\Omega_T)}\|\nabla \Phi\|_{L^1(\Omega_T)}\\
&+\tau \bar \gamma^2\left\|\frac{r_\tau}{1-\bar \gamma  \rho_\tau} \nabla\rho_\tau\right\|_{L^2(\Omega_T)}\|\nabla \Phi\|_{L^2(\Omega_T)}\\
&+\tau \|\tilde{u}_\tau\|_{L^2(0,T;H^1(\Omega))}\|\Phi\|_{L^2(0,T;H^1(\Omega))}\\
\leq & \, K\|\Phi\|_{L^2(0,T;H^1(\Omega))}.
\end{align*}
A similar estimate can be deduced for $b$ which concludes the proof.\qquad
\end{proof}

Even though the a priori estimates from Lemma \ref{lemma3} are enough to get boundedness for all terms in \eqref{equ1_tau} in $L^2(\Omega_T)$, the compactness results are not enough to identify the correct limits for $\tau \to 0$. 
From Lemma \ref{lemma3} we get that, as $\tau \to 0$
\begin{equation*}
\tau \tilde{u}_\tau, \tau \tilde{v}_\tau \to 0 \quad \text{ strongly in } L^2(0,T;H^1(\Omega)).
\end{equation*}
Together with Lemma \ref{lemma4}, we get a solution to
\begin{align}
\int_0^T \int_{\Omega} \begin{pmatrix}
\partial_t r\\ \partial_t b
\end{pmatrix}\cdot\begin{pmatrix}
 \Phi_1\\ \Phi_2
\end{pmatrix} \,d{\bf x}\,dt = \int_0^T \int_{\Omega} \begin{pmatrix}
J_r\\ J_b
\end{pmatrix}\cdot \begin{pmatrix}
\nabla \Phi_1\\\nabla \Phi_2
\end{pmatrix}\,d{\bf x}\,dt,
\end{align} 
where 
\begin{align}
(1-\bar \gamma \rho_\tau)\nabla  r_\tau+(\bar \alpha+\bar \gamma) r_\tau \nabla \rho_\tau + r_\tau\nabla V_r+\bar \gamma \nabla(V_b-V_r) r_\tau  b_\tau+\frac{\tau\bar \gamma^2r_\tau}{1-\bar \gamma\rho_\tau}\nabla \rho_\tau \rightharpoonup J_r, \label{limit1}\\
(1-\bar \gamma \rho_\tau)\nabla  b_\tau+(\bar \alpha+\bar \gamma) b_\tau \nabla \rho_\tau+ b_\tau\nabla V_b+\bar \gamma \nabla(V_r-V_b) r_\tau  b_\tau+\frac{\tau\bar \gamma^2 b_\tau}{1-\bar \gamma\rho_\tau}\nabla \rho_\tau\rightharpoonup J_b,\label{limit2}
\end{align}
weakly in  $L^2(\Omega_T)$.

In order to identify the limit terms, we multiply equation \eqref{limit1} by $(1-\overline{\gamma}\rho)$.

\begin{lemma}\label{lemma5}
For $\tau\to 0$, we have
\bigskip
\begin{compactenum}[(i)]
\item$(1-\bar \gamma \rho_\tau)^2\nabla  r_\tau \rightharpoonup(1-\bar \gamma \rho)^2\nabla  r $ weakly in  $L^2(\Omega_T)$\label{1},\\
\item$(1-\bar \gamma \rho_\tau)(\bar \alpha+\bar \gamma) r_\tau \nabla \rho_\tau\rightharpoonup(1-\bar \gamma \rho)(\bar \alpha+\bar \gamma) r \nabla \rho$ weakly in  $L^2(\Omega_T)$\label{2},\\ 
\item$(1-\bar \gamma \rho_\tau)r_\tau\nabla V_r \to (1-\bar \gamma \rho)r\nabla V_r$ strongly in  $L^2(\Omega_T)$\label{3},\\
\item$(1-\bar \gamma \rho_\tau)\bar \gamma \nabla(V_b-V_r) r_\tau  b_\tau \to (1-\bar \gamma \rho)\bar \gamma \nabla(V_b-V_r) r b$ strongly in  $L^2(\Omega_T)$\label{4},\\
\item $(1-\bar \gamma \rho_\tau)\frac{\tau\bar \gamma^2r_\tau}{1-\bar \gamma\rho_\tau}\nabla \rho_\tau =\tau\bar \gamma^2r_\tau\nabla \rho_\tau\to 0$ strongly in  $L^2(\Omega_T)$\label{5}.
\end{compactenum}
\end{lemma} 

\begin{proof}
The estimates from Lemma \ref{lemma3} and Lemma \ref{lemma4} allow us to use Aubin's lemma to deduce the existence of a subsequence (not relabeled) such that, as $\tau \to 0$:
\begin{equation}\label{conv_rho1}
\rho_\tau \to \rho \quad \text{ strongly in } L^2(\Omega_T).
\end{equation}
This implies
\begin{equation}\label{conv_rho}
1-\overline{\gamma} \rho_\tau \to 1-\overline{\gamma}\rho \quad \text{ strongly in } L^2(\Omega_T).
\end{equation}
Note that the $L^{\infty}$ bounds for $b_\tau$ and $r_\tau$ imply that, up to a subsequence,
\begin{equation}\label{conv_rb}
r_\tau\rightharpoonup r, \quad b_\tau\rightharpoonup b \quad \text{ weakly}^*\text{ in } L^{\infty}(\Omega_T).
\end{equation} 
With the help of a generalized version of Aubin-Lions Lemma (see Lemma 7 in \cite{zamponi2015analysis}), we also get strong convergence of the terms $(1-\overline{\gamma} \rho_\tau )r_\tau$ and $(1-\overline{\gamma} \rho_\tau )r_\tau b_\tau$. 
The lemma states that if \eqref{inequ3}, \eqref{conv_rho}, \eqref{conv_rb} and 
\begin{equation}
\|(1-\overline{\gamma}\rho_\tau)\,g\|_{L^2(0,T;H^1(\Omega))}\leq K\quad\text{ for } g\in \{1,r_\tau,b_\tau\}
\end{equation} 
hold, then we have strong convergence up to a subsequence for all $f=f(r_\tau,b_\tau)\in C^0(\mathcal{S};\mathbb{R}^2)$ of
\begin{equation}\label{conv_rhorb}
(1-\overline{\gamma}\rho_\tau)f(r_\tau,b_\tau) \to (1-\overline{\gamma}\rho)f(r,b) \quad \text{ strongly in } L^2(\Omega_T), 
\end{equation}
as $\tau \to 0$.

Applying \eqref{conv_rhorb} with $f(r_\tau,b_\tau)=r_\tau$, we get  
\begin{equation}\label{aubin1}
(1-\overline{\gamma}\rho_\tau)\,r_\tau\to (1-\overline{\gamma}\rho)\,r \quad \text{ strongly in }L^2(\Omega_T).
\end{equation}
Writing \eqref{1} as
\begin{align*}
(1-\overline{\gamma}\rho_\tau)^2\nabla r_\tau=(1-\overline{\gamma}\rho_\tau)\nabla((1-\overline{\gamma}\rho_\tau)r_\tau)-(1-\overline{\gamma}\rho_\tau)r_\tau\nabla(1-\overline{\gamma}\rho_\tau),
\end{align*}
and using the $L^{\infty}$ bounds together with the bounds in Lemma \ref{lemma3} to get $L^2$ bounds for  
$\nabla ((1-\overline{\gamma}\rho_\tau) r_\tau)=\nabla(1-\overline{\gamma}\rho_\tau)r_\tau+2\sqrt{r_\tau}\sqrt{1-\overline{\gamma}\rho_\tau}\sqrt{1-\overline{\gamma}\rho_\tau}\nabla\sqrt{r_\tau}$,
we can deduce
\begin{equation*}
(1-\bar \gamma \rho_\tau)^2\nabla  r_\tau \rightharpoonup(1-\bar \gamma \rho)^2\nabla  r \quad  \text{ weakly in } L^2(\Omega_T).
\end{equation*}
The convergence of \eqref{2} follows from the $L^{\infty}$ bounds, the a priori estimate \eqref{apriori2} as well as from the convergences \eqref{conv_rho1} and \eqref{aubin1}.

The strong convergences of \eqref{3} and \eqref{4} can be shown by applying \eqref{aubin1} in \eqref{3} and the generalized Aubin-Lions lemma with $f(r_\tau,b_\tau)=r_\tau b_\tau$ in \eqref{4}.

Finally, as $ r_\tau \nabla \rho_\tau$ is bounded in $L^2(\Omega_T)$ and $\tau \to 0$, we can deduce \eqref{5}.
\end{proof}
Analogous results hold for equation \eqref{limit2} which allows us to perform the limit $\tau \to 0$ giving a weak solution to system \eqref{theorem1_1}.

The only thing which remains to verify is the entropy inequality \eqref{theorem1_2}. Since $E$ is convex and continuous, it is weakly lower semi-continuous. Because of the weak convergence of $(r_\tau(t),b_\tau(t))$,
\[\int_\Omega \tilde{h}(r(t),b(t))\,d{\bf x}\leq \liminf_{\tau\to 0}\int_\Omega \tilde{h}(r_\tau (t),b_\tau (t))\,d{\bf x}\quad \text{ for a.e. } t>0.\]
We cannot expect the identification of the limit of $\sqrt{1-\rho_\tau}\nabla \sqrt{r_\tau}$, but employing \eqref{conv_rb} with $f(r,b)=\sqrt{r}$, we get 
\[(1-\overline{\gamma}\rho_\tau) \sqrt{r_\tau} \to (1-\overline{\gamma}\rho)\sqrt{r} \quad \text{ strongly in }L^2(\Omega_T)\]
with analogous convergence results for $r$ being replaced by $b$. 
Because of the $L^{\infty}$-bounds and the bounds in \eqref{lemma3}, we obtain $\nabla((1-\overline{\gamma}\rho_\tau)\sqrt{r_\tau}), \nabla((1-\overline{\gamma}\rho_\tau)\sqrt{b_\tau})\in L^2(\Omega_T)$, which implies 
\begin{align}\label{equ12}
\begin{aligned}
(1-\overline{\gamma}\rho_\tau) \sqrt{r_\tau} &\rightharpoonup (1-\overline{\gamma}\rho)\sqrt{r} \quad \text{ weakly in }L^2(0,T;H^1(\Omega)),\\
(1-\overline{\gamma}\rho_\tau) \sqrt{b_\tau} &\rightharpoonup (1-\overline{\gamma}\rho)\sqrt{b} \quad \text{ weakly in }L^2(0,T;H^1(\Omega)).
\end{aligned}
\end{align}
The $L^{\infty}$-bounds, \eqref{equ12} and the fact that
\[\nabla (1-\overline{\gamma}\rho_\tau)\rightharpoonup \nabla (1-\overline{\gamma}\rho) \quad \text{ weakly in } L^2(\Omega_T),\]
imply  that both
\[(1-\overline{\gamma}\rho_\tau)^2\nabla \sqrt{r_\tau}=(1-\overline{\gamma}\rho_\tau)\nabla ((1-\overline{\gamma}\rho_\tau) \sqrt{r_\tau})-(1-\overline{\gamma}\rho_\tau)\sqrt{r_\tau}\nabla (1-\overline{\gamma}\rho_\tau) \]
and
\[(1-\overline{\gamma}\rho_\tau)^2\nabla \sqrt{b_\tau}=(1-\overline{\gamma}\rho_\tau)\nabla ((1-\overline{\gamma}\rho_\tau) \sqrt{b_\tau})-(1-\overline{\gamma}\rho_\tau)\sqrt{b_\tau}\nabla (1-\overline{\gamma}\rho_\tau)\] 
converge weakly in $L^1$ to the corresponding limits. The $L^2$ bounds imply also weak convergence in $L^2$:
\begin{align*}
(1-\overline{\gamma}\rho_\tau)^2\nabla \sqrt{r_\tau}&\rightharpoonup (1-\overline{\gamma}\rho)^2\nabla \sqrt{r}\quad \text{ weakly in }L^2(\Omega_T),\\
(1-\overline{\gamma}\rho_\tau)^2\nabla \sqrt{b_\tau}&\rightharpoonup (1-\overline{\gamma}\rho)^2\nabla \sqrt{b}\quad \text{ weakly in }L^2(\Omega_T).
\end{align*}
As $1-\rho_\tau\geq (1-\rho_\tau)^4$, we can pass to the limit inferior $\tau\to 0$ in
\begin{align*}
\begin{aligned}
&\quad\int_{\Omega} \tilde{h}( r_\tau(T), b_\tau(T))\,d{\bf x}+\int_0^T\int_{\Omega} 2(1-\bar \gamma  \rho_\tau)^4|\nabla\sqrt{ r_\tau}|^2 +2(1-\bar \gamma  \rho_\tau)^4|\nabla\sqrt{ b_\tau}|^2+\frac{\bar \gamma}{2}|\nabla  \rho_\tau|^2\\
&\qquad \qquad +\frac{\tau^2}{2} \frac{\bar \gamma^5  \rho_\tau^2}{(1-\bar \gamma  \rho_\tau)^2}|\nabla  \rho_\tau |^2\,d{\bf x}\,dt+\tau\int_0^T R \left(\begin{pmatrix}
\tilde{u}_\tau\\\tilde{v}_\tau
\end{pmatrix},\begin{pmatrix}
\tilde{u}_\tau\\\tilde{v}_\tau
\end{pmatrix}\right)\,dt\leq  \int_{\Omega} \tilde{h}(r_0,b_0)\,dx\,dy+T C,
\end{aligned}
\end{align*}
attaining the entropy inequality \eqref{theorem1_2}.\qquad \endproof


\section{Conclusion}
Gradient flow techniques provide a natural framework to study the behavior of time evolving systems that are driven by an energy. This energy is decreasing along solutions as fast as possible, a property inherent in nature. Hence many partial differential equation models exhibit this structure. Most of these systems arise in the mean-field limit of a particle system, which has a gradient structure itself. 
Passing from the microscopic level to the macroscopic equations often relies on closure assumptions and approximations, which perturb the original gradient flow structure.

In this paper we studied a mean-field model for two species of interacting particles which was derived using the method of matched asymptotics in the case of low volume fraction.  This asymptotic expansions results in a cross-diffusion system which has a gradient flow structure up to a certain order. We therefore introduce the notion of asymptotic gradient flows for systems whose gradient flow structure is perturbed by higher order terms. We show that this 'closeness' to a classic gradient flow structure allows us to deduce existence and stability results for the perturbed or as we call them asymptotic gradient flow system. 

While the presented results on linear stability (Theorem \ref{linearstability}), well-posedness (Theorem \ref{wellposedness}) and existence of stationary solutions (Theorem \ref{theorem2}) also hold on unbounded domains, the proof of the global existence result in Section \ref{sec:existence} uses embeddings which do not hold on unbounded domains in general, e.g. $H^2(\Omega)$ is compactly embedded in $L^2(\Omega)$.   

The presented work is a first step towards the development of a more general framework for asymptotic gradient flows. It provides the necessary tools to understand the impact of high order perturbations on the energy dissipation as well as the behavior of solutions and opens interesting directions for future research.

\section*{Acknowledgments}
The work of MB was partially supported by the  German Science Foundation (DFG) through Cells-in-Motion Cluster of Excellence
(EXC 1003 CiM), M\"unster.
MTW and HR acknowledge financial support from the Austrian Academy of Sciences \"OAW via the New Frontiers Group NST-001. The authors thank the Wolfgang Pauli Institute (WPI) Vienna for supporting the workshop that lead to this work.


\bibliography{bibliography}

\begin{thebibliography}{10}

\bibitem{adams2011larg}
{\sc S.~Adams, N.~Dirr, M.~A. Peletier, and J.~Zimmer}, {\em From a
  large-deviations principle to the wasserstein gradient flow: A new
  micro-macro passage}, Communications in Mathematical Physics, 307 (2011),
  pp.~791--815.

\bibitem{amann1985global}
{\sc H.~Amann}, {\em Global existence for semilinear parabolic systems}, J.
  reine angew. Math, 360 (1985), pp.~47--83.

\bibitem{amann1989dynamic}
\leavevmode\vrule height 2pt depth -1.6pt width 23pt, {\em Dynamic theory of
  quasilinear parabolic systems}, Mathematische Zeitschrift, 202 (1989),
  pp.~219--250.

\bibitem{ambrosio2008gradient}
{\sc L.~Ambrosio, N.~Gigli, and G.~Savar{\'e}}, {\em Gradient flows: in metric
  spaces and in the space of probability measures}, Springer Science \&
  Business Media, 2008.

\bibitem{bendahmane2009conservative}
{\sc M.~Bendahmane, T.~Lepoutre, A.~Marrocco, and B.~Perthame}, {\em
  Conservative cross diffusions and pattern formation through relaxation},
  Journal de math{\'e}matiques pures et appliqu{\'e}es, 92 (2009),
  pp.~651--667.

\bibitem{Bruna:2012wu}
{\sc M.~Bruna and S.~J. Chapman}, {\em {Diffusion of multiple species with
  excluded-volume effects}}, J. Chem. Phys., 137 (2012),
  pp.~204116--204116--16.

\bibitem{Bruna:2012cg}
\leavevmode\vrule height 2pt depth -1.6pt width 23pt, {\em {Excluded-volume
  effects in the diffusion of hard spheres}}, Phys. Rev. E, 85 (2012),
  p.~011103.

\bibitem{MR2745794}
{\sc M.~Burger, M.~Di~Francesco, J.-F. Pietschmann, and B.~Schlake}, {\em
  Nonlinear cross-diffusion with size exclusion}, SIAM J. Math. Anal., 42
  (2010), pp.~2842--2871.

\bibitem{burger:2015uk}
{\sc M.~Burger, S.~Hittmeir, H.~Ranetbauer, and M.-T. Wolfram}, {\em Lane
  formation by side-stepping}, arXiv preprint arXiv:1507.08491,  (2015).

\bibitem{burger2012nonlinear}
{\sc M.~Burger, B.~Schlake, and M.~Wolfram}, {\em Nonlinear
  poisson--nernst--planck equations for ion flux through confined geometries},
  Nonlinearity, 25 (2012), p.~961.

\bibitem{carrillo2014gradient}
{\sc J.~A. Carrillo, S.~Lisini, and E.~Mainini}, {\em Gradient flows for
  non-smooth interaction potentials}, Nonlinear Analysis: Theory, Methods \&
  Applications, 100 (2014), pp.~122--147.

\bibitem{di2016nonlocal}
{\sc M.~Di~Francesco and S.~Fagioli}, {\em A nonlocal swarm model for
  predators--prey interactions}, Mathematical Models and Methods in Applied
  Sciences, 26 (2016), pp.~319--355.

\bibitem{Elliott:1996}
{\sc C.~M. Elliott and H.~Garcke}, {\em On the {C}ahn--{H}illiard equation with
  degenerate mobility}, SIAM Journal on Mathematical Analysis, 27 (1996),
  p.~404.

\bibitem{Evans199806}
{\sc L.~C. Evans}, {\em Partial Differential Equations (Graduate Studies in
  Mathematics, Vol. 19)}, Amer Mathematical Society, 1st~ed., 6 1998.

\bibitem{gilbarg2015elliptic}
{\sc D.~Gilbarg and N.~S. Trudinger}, {\em Elliptic partial differential
  equations of second order}, springer, 2015.

\bibitem{jungel2014boundedness}
{\sc A.~J{\"u}ngel and N.~Zamponi}, {\em Boundedness of weak solutions to
  cross-diffusion systems from population dynamics}, arXiv preprint
  arXiv:1404.6054,  (2014).

\bibitem{kato2013perturbation}
{\sc T.~Kato}, {\em Perturbation theory for linear operators}, vol.~132,
  Springer Science \& Business Media, 2013.

\bibitem{kipnis2013scaling}
{\sc C.~Kipnis and C.~Landim}, {\em Scaling limits of interacting particle
  systems}, vol.~320, Springer Science \& Business Media, 2013.

\bibitem{ol1968linear}
{\sc O.~A. Ladyzhenskai͡a, V.~A. Solonnikov, and N.~N. Ural'tseva}, {\em
  Linear and Quasi-linear Equations of Parabolic Type}, American mathematical
  society, 1968.

\bibitem{liero2013gradient}
{\sc M.~Liero and A.~Mielke}, {\em Gradient structures and geodesic convexity
  for reaction--diffusion systems}, Philosophical Transactions of the Royal
  Society of London A: Mathematical, Physical and Engineering Sciences, 371
  (2013), p.~20120346.

\bibitem{liero2015microscopic}
{\sc M.~Liero, A.~Mielke, M.~A. Peletier, and D.~Renger}, {\em On microscopic
  origins of generalized gradient structures}, arXiv preprint arXiv:1507.06322,
   (2015).

\bibitem{painter2009continuous}
{\sc K.~J. Painter}, {\em Continuous models for cell migration in tissues and
  applications to cell sorting via differential chemotaxis}, Bulletin of
  Mathematical Biology, 71 (2009), pp.~1117--1147.

\bibitem{Schlake:2011wr}
{\sc B.~Schlake}, {\em {Mathematical Models for Particle Transport: Crowded
  Motion}}, PhD thesis, Westf{\"a}lische Wilhelms-Universit{\"a}t M{\"u}nster,
  May 2011.

\bibitem{Simpson:2009gi}
{\sc M.~J. Simpson, K.~A. Landman, and B.~D. Hughes}, {\em {Multi-species
  simple exclusion processes}}, Physica A: Statistical Mechanics,  (2009).

\bibitem{zamponi2015analysis}
{\sc N.~Zamponi and A.~J{\"u}ngel}, {\em Analysis of degenerate cross-diffusion
  population models with volume filling}, in Annales de l'Institut Henri
  Poincare (C) Non Linear Analysis, Elsevier, 2015.

\bibitem{zinsl2015transport}
{\sc J.~Zinsl and D.~Matthes}, {\em Transport distances and geodesic convexity
  for systems of degenerate diffusion equations}, Calculus of Variations and
  Partial Differential Equations, 54 (2015), pp.~3397--3438.

\end{thebibliography}
\bibliographystyle{siam}

\label{lastpage}

\end{document}